\documentclass[reqno, 12pt]{amsart}

\usepackage{mathtools}
\usepackage[T1]{fontenc}
\usepackage{amsmath, amsthm, amssymb}
\usepackage[ansinew]{inputenc}
\usepackage{xcolor}
\usepackage{bbm}
\usepackage{mathrsfs}
\usepackage{cite}
\usepackage{xparse}
\usepackage{fullpage}
\usepackage{pdflscape}
\usepackage{stmaryrd}
\usepackage{multirow}
\usepackage{lscape}
\usepackage[colorlinks=true,allcolors=red]{hyperref}
\usepackage{graphicx} 
\usepackage{accents}
\newcommand*{\dt}[1]{%
  \accentset{\mbox{\large\bfseries .}}{#1}}
\newcommand*{\ddt}[1]{%
  \accentset{\mbox{\large\bfseries .\hspace{-0.25ex}.}}{#1}}

\newtheorem{theorem}{Theorem}[section]
\newtheorem{lemma}[theorem]{Lemma}
\newtheorem{prop}[theorem]{Proposition}

\theoremstyle{definition}

\theoremstyle{remark}
\newtheorem{remark}[theorem]{Remark}

\allowdisplaybreaks[4]

\numberwithin{equation}{section}

\usepackage{scalerel,stackengine}
\stackMath
\newcommand\reallywidehat[1]{
\savestack{\tmpbox}{\stretchto{
  \scaleto{
    \scalerel*[\widthof{\ensuremath{#1}}]{\kern.1pt\mathchar"0362\kern.1pt}
    {\rule{0ex}{\textheight}}
  }{\textheight}
}{2.4ex}}
\stackon[-6.9pt]{#1}{\tmpbox}
}
\makeatletter
\newcommand*{\rom}[1]{\expandafter\@slowromancap\romannumeral #1@}
\makeatother

\begin{document}

\title{On time periodic solutions to the conformal cubic wave equation on the Einstein cylinder}

\author{Athanasios Chatzikaleas}

\address{Laboratory Jacques-Louis Lions (LJLL), University Pierre and Marie Curie (Paris 6), 4 place Jussieu, 75252 Paris, France}
\email{athanasios.chatzikaleas@ljll.math.upmc.fr}

\dedicatory{}
\begin{abstract}
%
%
We consider the conformal wave equation on the Einstein cylinder with a defocusing cubic non-linearity studied in \cite{Bizon1}. Motivated by a method developed by Rostworowski-Maliborski \cite{13033186} on the existence of time periodic solutions to the spherically symmetric Einstein-Klein-Gordon system, we study perturbations around the zero solution as a formal series expansion and assume that the perturbations bifurcate from one mode. In the center of this work stands a rigorous proof on how one can choose the initial data to cancel out all secular terms in the resonant system. Interestingly, our analysis reveals that the only possible choice for the existence of time periodic solutions is when the error terms in the expansion are all proportional to the dominant one mode. Finally, we use techniques from ordinary differential equations and establish the existence of time periodic solutions for initial data proportional to the first mode of the linearized operator.
\end{abstract}

\maketitle
\tableofcontents
\noindent
\section{Introduction}
\subsection{Conformal cubic wave equation on the Einstein cylinder }
In this note, we are interested in the propagation of non-linear relativistic scalar waves in the Anti-de-sitter (AdS) spacetime in $1+3$ dimensions. The AdS spacetime $\left(\mathcal{M}^{1+3},g_{AdS} \right)$ is the unique maximally symmetric solution to the Einstein's equations in vacuum 
\begin{align*}
	R_{ \alpha \beta} -\frac{1}{2} g _{\alpha \beta} R + \Lambda g_{\alpha \beta} = 0
\end{align*}
with negative cosmological constant 
\begin{align*}
	\Lambda = - 3.
\end{align*} 
With respect to local spherical coordinates on the underlying manifold 
\begin{align*}
    (t,r,\theta,\phi) \in \mathcal{M}^{1+3}:=
    \mathbb{R} \times [0,\infty) \times [0,\pi] \times [0,2\pi)
\end{align*}
this solution reads
\begin{align*}
	g_{\text{AdS}} (t,r,\theta,\phi) = -\left( 1+r^2 \right) d t ^2 + \left( 1+r^2 \right)^{-1} d r^2 + r^2 \left(
	d \theta^2 + \sin^2(\theta) d \phi^2
	\right).
\end{align*}
Following a recent work of Bizo\'{n}-Craps-Evnin-Hunik-Luyten-Maliborski \cite{Bizon1}, we consider a real scalar field 
\begin{align*}
	u:(\mathcal{M}^{1+3},g_{\text{AdS}}) \longrightarrow \mathbb{R}
\end{align*}
which propagates in the $(1+3)-$dimensional AdS spacetime satisfying the semi-linear wave equation  
 \begin{align}\label{kurioseksisosi}
 	\left(\Box_{ g_{\text{AdS} } } - \frac{1}{6} R(g_{\text{AdS}} ) \right) u = \lambda u^3
 \end{align}
with a defocusing nonlinearity, meaning $\lambda >0$. Here, $\Box_{ g_{\text{AdS} } }$ stands for the wave operator 
and $R(g_{\text{AdS}} )$ denotes the scalar curvature both with respect to the AdS metric. It is convenient to compactify $\left( \mathcal{M}^{1+3},g_{\text{AdS}} \right)$ and introduce $\psi=\arctan(r)$. Now, the new spatial coordinates vary within a bounded region
\begin{align*}
	(t,\psi,\theta,\phi) \in  \mathcal{M}^{1+3} \simeq
    \mathbb{R} \times \left[0,\frac{\pi}{2} \right) \times [0,\pi] \times [0,2\pi)
\end{align*}
and the AdS solution takes the form
\begin{align}\label{conformal}
	g_{\text{AdS}}(t,\psi,\theta,\phi) = \Omega^2(\psi) g_{ \mathcal{E} }(t,\psi,\theta,\phi),\quad \Omega(\psi):=\frac{1}{\cos (\psi)}
\end{align}
implying that it is conformal to half of the full Einstein static universe
\begin{align*}
	(t,\psi,\theta,\phi) \in  \mathcal{E}^{1+3}:=
    \mathbb{R} \times \left[0,\pi \right] \times [0,\pi] \times [0,2\pi)
\end{align*}
endowed with the metric
\begin{align*}
g_{\mathcal{E}}(t,\psi,\theta,\phi):=	- dt^2 + d\psi ^2 + \sin ^2 (\psi) \left(
 d\theta^2 + \sin^2 (\theta) d \phi^2
\right).
\end{align*}
As in \cite{Bizon1}, one can use the conformal structure \eqref{conformal} together with the conformal change of variables $v:=\Omega u$ to simplify the equation. Specifically, one can write \eqref{kurioseksisosi} as
\begin{align*}
 	\left(\Box_{ \Omega^2  g_{\mathcal{E} } } - \frac{1}{6} R(\Omega^2  g_{\mathcal{E} }  ) \right) \left( \Omega^{-1} v \right) = \lambda \left( \Omega^{-1} v \right)^3 
\end{align*}
and use the identity
\begin{align*}
	 \Box_{ \Omega^2  g_{\mathcal{E} } } - \frac{1}{6} R(\Omega^2  g_{\mathcal{E} }  ) =
	  \Omega^{-3}  \left(\Box_{ g_{\mathcal{E} } } - \frac{1}{6} R( g_{\mathcal{E} }  ) \right)
\end{align*}
to infer
\begin{align*}
	\left(\Box_{ g_{\mathcal{E} } } - \frac{1}{6} R(g_{\mathcal{E} }  ) \right)  v = \lambda   v^3.
\end{align*}
Next, we factor out the parameter $\lambda$ just by replacing $v$ with $w / \sqrt{\lambda}$. We get
\begin{align*}
	\left(\Box_{ g_{\mathcal{E} } } - \frac{1}{6} R(g_{\mathcal{E} }  ) \right)  w =  w^3
\end{align*}
and compute
\begin{align*}
	\Box_{ g_{\mathcal{E} } }= - \partial_{t}^2 + \Delta_{(\psi,\theta,f)}^{\mathbb{S}^3}, \quad  R(g_{\mathcal{E}}) = 6,
\end{align*}
where
\begin{align*}
	\Delta_{(\psi,\theta,\phi)}^{\mathbb{S}^3} w = 
	\frac{1}{\sin^2(\psi)} 
	\Bigg[
	\partial_{\psi} \left(\sin^2(\psi) \partial_{\psi} w \right)+
	\frac{1}{\sin(\theta)}
	\left(
	 \partial_{\theta} \left(\sin(\theta) \partial_{\theta} w \right)+
	\frac{1}{\sin(\theta)} 
	\partial_{\phi}^2 w
	\right)
	\Bigg]
\end{align*}
stands for the Laplace-Beltrami operator on $\mathbb{S}^3$. We obtain
\begin{align*}
	- \partial_{t}^2 w(t,\psi,\theta,\phi)+ \Delta_{(\psi,\theta,\phi)}^{\mathbb{S}^3}w(t,\psi,\theta,\phi) - w(t,\psi,\theta,\phi) = w^3(t,\psi,\theta,\phi),
\end{align*}
for $(t,\psi,\theta,\phi) $ varying only within half of $\mathcal{E}^{1+3}$, namely 
\begin{align*}
(t,\psi,\theta,\phi) \in \mathbb{R} \times \left[0, \frac{\pi}{2} \right) \times [0,\pi] \times [0,2\pi). 	
\end{align*}
Now, we look for rotationally symmetric solutions of the form $w(t,\psi)$ as in \cite{Bizon1}. In order to have reflective boundary conditions at the conformal infinity $\{\psi = \frac{\pi}{2} \}$ we assume the Dirichlet condition $w\left(t,\frac{\pi}{2} \right)=0$ and extent the solution from the half to the full Einstein cylinder $\mathcal{E}^{1+3}$ using the reflection symmetry
\begin{align*}
	- w(t,\psi) = w(t,\pi - \psi).
\end{align*}
Next, we map
\begin{align*}
	- \partial_{t}^2 w(t,\psi) + \frac{1}{\sin^2(\psi)} 
	\partial_{\psi} \left(\sin^2(\psi) \partial_{\psi} w(t,\psi) \right) - w(t,\psi) = w^3(t,\psi)
\end{align*}
to the $1-$dimensional wave equation
\begin{align*}
	- \partial_{t}^2 f(t,\psi) +  \partial_{\psi}^2 f(t,\psi) = \frac{f^3(t,\psi)}{\sin^2(\psi)}
\end{align*}
via the transformation $f(t,\psi) = \sin(\psi) w(t,\psi)$ for all $(t,\psi) \in \mathbb{R} \times [0,\pi]$. Finally, we formally compute
\begin{align*}
	& \frac{f^3(t,\psi)}{\sin^2(\psi)} = \frac{f^3(t,0)}{\psi^2} + \frac{3f^2(t,0) \partial_{\psi}f(t,0)}{\psi} + \mathcal{O} \left(1 \right), \text{ for } \psi \longrightarrow 0+, \\
	& \frac{f^3(t,\psi)}{\sin^2(\psi)} = \frac{f^3(t,\pi)}{(\pi-\psi)^2} - \frac{3f^2(t,\pi) \partial_{\psi}f(t,\pi)}{\pi-\psi} + \mathcal{O} \left(1 \right), \text{ for } \psi \longrightarrow \pi-
\end{align*}
and hence, in order to ensure the regularity of the solutions, we impose the Dirichlet boundary conditions
\begin{align*}
	f(t,0) = f(t,\pi) = 0,
\end{align*}
for all times. 
\subsection{Structure of the paper}
We consider the initial value problem which consists of the cubic wave equation on the Einstein cylinder
\begin{align}\label{maineq11}
	 -  \partial_{t}^2 f (t,\psi) +  \partial_{\psi}^2 f (t,\psi) = \frac{f^3(t,\psi)}{\sin^2(\psi)}, \quad (t,\psi) \in \mathbb{R} \times \left(0,\pi \right), 
\end{align}
subject to Dirichlet boundary conditions
\begin{align}\label{maineq22}
	 f(t,0) = f(t,\pi) = 0, \quad t \in \mathbb{R}.
\end{align}  
Motivated by \cite{Bizon1} where the existence of stationary states to \eqref{maineq11}-\eqref{maineq22} is discussed as well as by \cite{MR4026950}, our work is directed towards a twofold aim. From the one hand, we investigate the non-linear stability of the zero solution to \eqref{maineq11}-\eqref{maineq22} and on the other hand the existence of time periodic solutions. First, we use standard techniques together with Hardy-Sobolev inequalities to establish the local well-posedness to \eqref{maineq11}-\eqref{maineq22} in the energy space $H^{1}[0,\pi] \times L^2[0,\pi]$ (Appendix \ref{LWP}) so that the present work is self-contained. Next, we discuss recent advances about the existence of periodic and quasi-periodic solutions (section \ref{TPingeneral}) and in particular we explain the method of Rostworowski-Maliborski \cite{13033186} developed for the Einstein-Klein-Gordon system within spherical symmetry (section \ref{RostworowskiMaliborski}) 
which shows how the initial data contribute to the production of secular terms, i.e. terms which spoil the periodicity of the solutions. To begin our analysis, we consider perturbations bifurcating from one mode using a formal series expansion 
	\begin{align*}
		f(t,\psi)
		 =\epsilon   \cos \left( \left(1+\mathcal{O}(\epsilon^2) \right) t \right) e_{\gamma}(\psi) + \mathcal{O} \left(\epsilon^2 \right)
	\end{align*}
	where $e_{i}$ for $ i=0,1,2,\dots$ are the eigenfunctions associated to the linearized operator coupled to Dirichlet boundary conditions and for simplicity we choose $\gamma=0$. In addition, we use the method of Rostworowski-Maliborski \cite{13033186} to see how secular terms occur for the non-linear problem \eqref{maineq11} (section \ref{Preliminaries}). Then, we rigorously prove how one can prescribe the initial data to cancel all secular terms to \eqref{maineq11}-\eqref{maineq22} and finally we establish the existence of time periodic solutions (section \ref{Statement}). 
\subsection{Statement of the main results}	Specifically, we work towards proving the following two results. We refer the reader to Theorem \ref{maintheorem} and Theorem \ref{maintheorem2} respectively for more detail versions of these results.
\begin{theorem}
	Let $f$ be a solution to the boundary value problem \eqref{maineq11}-\eqref{maineq22} with initial data $\left(f(0,\cdot),\partial_{t}f(0,\cdot) \right)$ where $\partial_{t}f(0,\cdot) =0$. We assume that $f$ admits the representation as a formal power series as in \eqref{series}. Then, the corresponding resonant system admits no secular term if and only if the initial data are all proportional to $e_{0}$. 
\end{theorem}
In particular, we prove that time periodic solutions in the neighbourhood of the $e_{0}$ mode arise necessarily from such initial data.
\begin{theorem}
All initial data of the form $\left( f(0,\cdot), \partial_{t}f(0,\cdot) \right)=\left(G(0) e_0,0 \right)$ evolve into time periodic solutions.
\end{theorem}

\subsection{Acknowledgments}
The author would like to express his sincere gratitude to Professor Jacques Smulevici for very useful communications, comments and insights. Also, the author gratefully acknowledges the support of the ERC grant 714408 GEOWAKI, under the European Union's Horizon 2020 research and innovation program.
\section{Time periodic solutions to non-linear wave equations} \label{TPingeneral}
\subsection{Discussion}
The existence of periodic and quasi-periodic solutions to nonlinear wave equations has attracted a lot of attention in the last years in the mathematical analysis community. Such studies were initiated by Rabinowitz \cite{MR470378} and Brezis-Coron-Nirenberg \cite{MR586417} who considered $1-$dimensional non-linear wave equations of the form
\begin{align*}
-\partial_{t}^2 f(t,x)+ \partial_{x}^2 f(t,x) = N(x,f(t,x)),
\end{align*}
together with boundary conditions as well as reasonable assumptions on the non-linearity, and proved the existence and regularity of periodic solutions. Since then, extensive studies have been made concentrated around the mechanism responsible for the existence of time periodic solutions for various partial differential equations. For example, the reader can consult the works of Baldi-Berti-Haus-Montalto-Alazard-Baldi for gravity water waves \cite{MR3867631, MR4062430, MR3625065, MR3356988}, the works of Wyane-Kuksin-Berti-Corsi-Precesi-Biasco-Bolle-Bourgain for the non-linear wave equations \cite{MR911772, MR1040892, MR3312439, MR3073240, MR2967117, MR2592290, MR2038735}, the works of Berti-Bolle-Corsi-Procesi-Delort for the non-linear Schr\"{o}dinger equation \cite{MR2998835, MR3312439, MR2813578, MR2901562}, the work of Montalto for the forced Kirchhoff equation \cite{MR3603787}, the work of Baldi-Berti-Montalto-Lax on the KdV equation \cite{MR0344645, MR404889, MR369963, MR3237812}, the work of on the Ambrose-Wilkening for the Benjamin-Ono equation \cite{MR2639896}, among others. The main difficulty to establish the existence of time periodic solutions stems from the presence of small divisors in the perturbation series around equilibrium points. Such a problem can be particularly difficult especially in the case where the non-linearity contains derivatives of the dynamical variable. The key ingredients used to overcame this difficulty are the KAM theory and the Nash-Moser implicit function theorem \cite{MR2070057, MR3502157, MR1986317, MR2496651, MR1754991, MR1857574, MR2842962}. \\ \\
Some of the methods already existing in the substantial literature can be also used to obtain the existence of time periodic solutions to \eqref{maineq11}-\eqref{maineq22}, possibly with modifications. For example, one can use the method of Bambusi-Paleari \cite{MR1819863} who considered wave equations of the form
\begin{align*}
	-\partial_{t}^2 f (t,x)+\partial_{x}^2 f (t,x) = N(f(t,x)), \quad (t,x) \in \mathbb{R} \times (0,\pi)
\end{align*}
with non-linearities satisfying
\begin{align*}
	N \in C^4,\quad N(0)=N^{\prime}(0)=N^{\prime \prime}(0)=0, \quad N^{\prime \prime \prime}(0) \neq 0
\end{align*}
together with the Dirichlet boundary conditions 
\begin{align*}
	 f(t,0) = f(t,\pi) = 0, \quad t \in \mathbb{R}
\end{align*} 
and constructed some families of time periodic solutions with small amplitude using averaging methods meaning suitable maps from the configurations space to itself. Specifically, their method relies on the Lyapunov-Schmidt decomposition \cite{MR1239318} and a non-degeneracy condition. Most importantly, their method focuses on the fact that each non-degeneracy zero of such maps implies a whole family of small amplitude time periodic solutions to the wave equation above revealing an interesting correspondence between the Lyapunov-Schmidt decomposition and averaging theory. Although the non-linearity they considered depends only of $f$ and not on $(x,f)$, one can still adapt their method to obtain time periodic solutions to \eqref{maineq11}-\eqref{maineq22}. However, we emphasize that the regularity of the nonlinear term in \eqref{maineq11} prevent from a direct application of the method described above. In addition, one can also use the method of Berti-Bolle \cite{MR2248834} who considered one dimensional wave equations of the form
\begin{align*}
\begin{cases}
	 -\partial_{t}^2 f (t,x)+\partial_{x}^2 f (t,x) = \alpha_{3}(x) f^3(t,x) + \mathcal{O}(f^4), \\
	  f(t,0) = f(t,\pi) = 0
\end{cases}
\end{align*}
and proved the existence of small amplitude periodic in time solutions with period $\frac{2 \pi}{\omega}$ for any frequency $\omega$ in a Cantor-like set. This Cantor set arises from non-resonance conditions imposed on the frequency $\omega$ and is essential as it allows to overcame the small-divisors problem. However, in this work, we do not address the convergence of the perturbative series expansion and hence such a Cantor set does not occur. To establish the existence of time periodic solutions, we first focus on how one can rigorously cancel out all secular terms for the initial value problem \eqref{maineq11}-\eqref{maineq22} by prescribing the initial data. To do so, we use a method of Rostworowski-Maliborksi \cite{13033186} developed for the Einstein-Klein-Gordon equation. 

\subsection{The method of Rostworowski-Maliborski}\label{RostworowskiMaliborski}
Rostworowski-Maliborksi \cite{13033186} considered the Einstein-Klein-Gordon system which consists of the wave equation 
\begin{align}\label{mainRM1}
	\Box_{g} \phi = 0
\end{align}
for a real scalar field $\phi: \left( \mathcal{M},g \right) \longrightarrow \mathbb{R}$ coupled to the Einstein equations
\begin{align}\label{mainRM2}
	R_{ \alpha \beta} -\frac{1}{2} g _{\alpha \beta} R + \Lambda g_{\alpha \beta} = 8 \pi \left ( \partial _{\alpha} \phi \partial _{\beta} \phi - \frac{1}{2} g_{\alpha \beta} (\partial \phi)^2 \right )
\end{align}
with negative cosmological constant
\begin{align*}
\Lambda = -\frac{d(d-1)}{2 l^2 }, \quad  l \neq 0.
\end{align*}
In particular, they assumed that the scalar field $\phi$ propagates within $(1+d)-$dimensional spacetimes $\left( \mathcal{M}, g \right)$ for any $d\geq 2$ satisfying the spherically symmetric ansatz
\begin{align*}
	g(t,x,\omega) = \frac{l^2}{\cos ^2 (x)} \left( - \frac{ A(t,x)}{ e^{2\delta(t,x)}} dt^2 + \frac{1}{A(t,x)} dx^2+ \sin ^2 (x) d \omega ^2 \right),
\end{align*}
where the local coordinates read
\begin{align*}
	(t,x,\omega) \in  \mathcal{M}:= \mathbb{R} \times \left[0,\frac{\pi}{2} \right) \times \mathbb{S}^{d-1}.
\end{align*}
Then, one can define the auxiliary variables
\begin{align*}
	\Phi (t,x) = \partial_{x} \phi (t,x),~~~\Pi (t,x) = \frac{1}{A(t,x)e^{-\delta (t,x)} } \partial_{t} \phi(t,x)
\end{align*}
and express \eqref{mainRM1} equivalently as
\begin{align}
	\partial_{t} \Phi(t,x) &= \partial_{x} \left( A(t,x)e^{-\delta(t,x)} \Pi(t,x) \right), \label{EKG1}   \\  
	 \partial_{t} \Pi(t,x) &= -\reallywidehat{L} \left[ A(t,x)e^{-\delta (t,x)} \Phi (t,x) \right],  \label{EKG2}
\end{align}
where
\begin{align*}
	\reallywidehat{L}[g](x):= -\frac{1}{\tan^2(x)} \partial_{x} ( \tan^2(x) g(x)),
\end{align*} 
coupled to the Einstein equations \eqref{mainRM2} which are given by
\begin{align}
		(1-A(t,x))e^{-\delta(t,x)} &= \frac{\cos ^3 (x)}{\sin(x)} \int_{0}^{x} e^{-\delta(t,y)} \left(\Phi^2(t,y) + \Pi^2(t,y) \right) (\tan(y))^2 dy\label{EKG3}  \\ 
	-\delta(t,x) &= \int_{0}^{x} \left(\Phi^2(t,y) + \Pi^2(t,y) \right) \sin(y) \cos(y) dy.\label{EKG4}
\end{align}
Specifically, Rostworowski-Maliborksi \cite{13033186} provided reliable numerical evidence indicating that time-periodic solutions to \eqref{EKG1}-\eqref{EKG2}-\eqref{EKG3}-\eqref{EKG4} exist for non-generic initial data (islands of stability) and in addition constructed these solutions using both nonlinear perturbative expansions and numerical methods. Their work is closely related and builds upon an earlier numerical study of Bizo\'{n}-Rostworowski \cite{11043702} who considered \eqref{mainRM1}-\eqref{mainRM2} under the same spherically symmetric ansatz for $(1+3)-$dimensional spacetimes and established strong numerical results which show that the AdS solution 
\begin{align*}
	g_{AdS}(t,x,\omega) = \frac{l^2}{\cos ^2 (x)} \left(
	- dt^2 +  dx^2+ \sin ^2 (x) d \omega ^2 
	\right),
\end{align*}
to the Einstein equations (although linearly stable) is in fact nonlinearly unstable against the formation of a black hole under arbitrarily small and generic perturbations. The strong numerical evidence of Bizo\'{n}-Rostworowski \cite{11043702} support a conjecture on the instability of the AdS spacetime firstly announced by Dafermos \cite{DafermosTalk} and Dafermos-Holzegel \cite{DafermosHolzegel} in 2006. \\ \\

Rostworowski-Maliborksi \cite{13033186} constructed the desired periodic solutions to \eqref{EKG1}-\eqref{EKG2}-\eqref{EKG3}-\eqref{EKG4} using nonlinear perturbative expansions as well as numerical simulations. In the following lines we explain their first approach. To begin with, they sought solutions $(\Phi,\Pi, A, \delta)$ which are all close to the AdS solution $(0,0, 1, 0)$ and (although the series may not converge) considered the following perturbative expansions
\begin{align}
	&\Phi(t,x) =\sum_{\lambda=0}^{\infty} \psi_{2\lambda+1} (\tau,x)\epsilon^{2\lambda+1} = \psi_{1} (\tau,x)\epsilon + \psi_{3} (\tau,x)\epsilon^{3}+\psi_{5} (\tau,x)\epsilon^{5}+\dots,\label{series1} \\ 
	& \Pi(t,x)=\sum_{\lambda=0}^{\infty} \sigma_{2\lambda+1} (\tau,x)\epsilon^{2\lambda+1}=\sigma_{1} (\tau,x)\epsilon + \sigma_{3} (\tau,x)\epsilon^{3}+\sigma_{5} (\tau,x)\epsilon^{5}+\dots, \label{series2} \\
	& A(t,x) e^{-\delta(t,x)} =\sum_{\lambda=0}^{\infty} \xi_{2\lambda} (\tau,x)\epsilon^{2\lambda}= \xi_{0}(\tau,x) +\xi_{2}(\tau,x) \epsilon^2+\xi_{4}(\tau,x) \epsilon^4+\dots,\nonumber  \\
	& e^{-\delta(t,x)}=\sum_{\lambda=0}^{\infty} \zeta_{2\lambda} (\tau,x)\epsilon^{2\lambda}=\zeta_{0}(\tau,x) +\zeta_{2}(\tau,x) \epsilon^2+\zeta_{4}(\tau,x) \epsilon^4+\dots \nonumber \\
	& \tau = \Omega t, \quad
	\Omega  = \sum_{\lambda=0}^{\infty} \theta_{2\lambda} \epsilon^{2\lambda} = \theta_{0}+ 
	\theta_{2} \epsilon^{2}+\cdots.
\end{align}
where $\psi_{2\lambda+1}$,$\sigma_{2\lambda+1}$,$\xi_{2\lambda}$ and $\zeta_{2\lambda}$ are all periodic in time. With a slight abuse of notation we use the same letters to denote the dynamical variables with respect to the $(\tau,x)$ and $(t,x)$.
One can see \cite{13033186, SecondPublication} that the operator which governs the solutions to the linearized equation in \eqref{EKG1}-\eqref{EKG2} is given by
\begin{align*}
	L[f](x):= -\frac{1}{\tan^2(x)} \partial_{x} ( \tan^2(x) \partial_{x} f(x)).
\end{align*}
The linearized operator $L$ is self-adjoint with respect to the weighted inner product
\begin{align} \label{inner}
	(f|g):=\int _{0}^{\frac{\pi}{2}} f(x)g(x) \tan^{2}(x) dx,
\end{align}
its spectrum subject to Dirichlet boundary conditions is given by
\begin{align*}
	\omega^2_{j}:=(3+2j)^2,~j=1,2,\dots
\end{align*}
and the eigenfunctions are weighted Jacobi polynomials
\begin{align*}
	e_{j}(x):=2\frac{\sqrt{j! (j+2)!}}{\Gamma (j+\frac{3}{2})} \cos^3(x) P^{\frac{1}{2},\frac{3}{2}}_{j}(\cos(2x)),\quad x \in \left[0,\frac{\pi}{2} \right], \quad j=0,1,\dots.
\end{align*}
Alternatively (Lemma 2.2 in \cite{SecondPublication}), we have
\begin{align*}
	e_{i}(x) = \frac{2}{\sqrt{\pi}} \frac{ 1}{\sqrt{\omega_{i}^2-1}}  \left( \omega_{i} \frac{\sin\left( \omega_{i} x \right)}{\tan(x)} - \cos \left( \omega_{i}x \right) \right), \quad x \in \left[0,\frac{\pi}{2} \right], \quad j=0,1,\dots.
\end{align*}
In addition, we have the following orthogonality properties
\begin{align*}
	(e_{i}|e_{j}) = \delta_{i,j},\quad (e_{i}^{\prime}|e_{j}^{\prime}) = \omega_{i}^2 \delta_{i,j}
\end{align*}
and both $\{e_{j}\}$ and $\{ \frac{e^{\prime}_{j}}{\omega_{j}}\}$ form an orthonormal basis for the $L^2$ with respect to the inner product \eqref{inner}. All these results can be found in \cite{SecondPublication} (Lemma 2.2) and in \cite{MR2430631} (appendix). 
After solving at the linear level, we obtain the initial values
\begin{align*}
     \theta_{0} = \omega_{0}, \quad
	 \xi_{0}(\tau,x) = 1,\quad
	 \zeta_{0}(\tau,x) = 1
\end{align*}
and
\begin{align*}
\begin{cases}
	\omega_{\gamma} \partial_{\tau} \psi_{1}(\tau,x) - \partial_{x} \sigma_{1} (\tau,x) = 0, \\
	\omega_{\gamma} \partial_{\tau} \sigma_{1}(\tau,x)+ \reallywidehat{L} \left[ \psi_{1}(\tau,x) \right]=0.
\end{cases}
\end{align*}
From the set of all eigenfunctions $\{e_{i}:i=0,1,2,\dots\}$ to the linearized operator, we choose a dominant mode $e_{\gamma}$ for some $\gamma \in \{0,1,2,\dots \}$. For simplicity we choose $e_{0}$ and pick
\begin{align*}
\begin{cases}
	\psi_{1} (\tau,x) = \cos(\tau) e^{\prime}_{0}(x),\\
	\sigma_{1} (\tau,x) =  -\omega_{0} \sin(\tau) e_{0}(x).
\end{cases}
\end{align*}
First, we compute the density
\begin{align*}
	\Phi^2(t,x)+\Pi^2(t,x) = \sum_{\lambda=1}^{\infty} r_{2\lambda} \epsilon^{2\lambda}
\end{align*}
where, for all $\lambda = 0,1,2,\dots$,
\begin{align*}
	r_{2(\lambda+1)}(\tau,x) = \sum_{\substack{\mu,\nu=0 \\ \mu+\nu=\lambda }}^{\lambda}
	\left( \psi_{2\mu+1}(\tau,x)\psi_{2\nu+1}(\tau,x) +
	 \sigma_{2\mu+1}(\tau,x)\sigma_{2\nu+1} (\tau,x)
	\right).
\end{align*}
Next, we substitute these expressions into \eqref{EKG1}-\eqref{EKG2}-\eqref{EKG3}-\eqref{EKG4}, collect terms of the same order in $\epsilon$ and obtain a hierarchy of equations
\begin{align*}
\begin{dcases}
	\omega_{\gamma} \partial_{\tau} \psi_{2\lambda+1} (\tau,x) - \partial_{x} \sigma_{2\lambda+1} (\tau,x)= 
	\sum_{\substack{\mu,\nu=0 \\ \mu+\nu=\lambda \\ (\mu,\nu) \neq (\lambda,0) }}^{\lambda}
	\Big (
	- \omega_{\gamma,2\nu} \partial_{\tau} \psi_{2\mu+1} (\tau,x) + \partial_{x} \left( \xi_{2\nu}(\tau,x) \sigma_{2\mu+1}(\tau,x) \right)
	 \Big ), 
	\\
	 \omega_{\gamma} \partial_{\tau} \sigma_{2\lambda+1} (\tau,x) 
	+ \reallywidehat{L}  \left[ \psi_{2\lambda+1}(\tau,x) \right]
	= -  \sum_{\substack{\mu,\nu=0 \\ \mu+\nu=\lambda \\ (\mu,\nu) \neq (\lambda,0) }}^{\lambda} 
	\Big(
	\omega_{\gamma,2\nu}\partial_{\tau} \sigma_{2\mu+1}(\tau,x)+  \reallywidehat{L}  \left[ \xi_{2\nu}(\tau,x) \psi_{2\mu+1}(\tau,x) \right]
	\Big),\\
	  \xi_{2(\lambda+1)} (\tau,x) = \zeta_{2(\lambda+1)} (\tau,x) - \frac{\cos^3(x)}{\sin(x)} 
	\sum_{\substack{\mu,\nu=0 \\ \mu+\nu=\lambda  }}^{\lambda} \int_{0}^{x} r_{2(\mu+1)}(\tau,y) \zeta_{2\nu}(\tau,y) \tan^2(y) dy,\\
	 \zeta_{2(\lambda+1)} (\tau,x) =  
	\sum_{\substack{\mu,\nu=0 \\ \mu+\nu=\lambda }}^{\lambda} \int_{0}^{x} r_{2(\mu+1)}(\tau,y) \zeta_{2\nu}(\tau,y)\sin(y)\cos(y)dy
\end{dcases}
\end{align*}
for all $\lambda=0,1,2,\dots$. Since both $\{e_{j}\}$ and $\{ \frac{e^{\prime}_{j}}{\omega_{j}}\}$ form an orthonormal basis for $L^2([0,\frac{\pi}{2}])$ with respect to the inner product \eqref{inner}, we expand the coefficients $\psi_{2\lambda+1},\sigma_{2\lambda+1},\xi_{2\lambda},\zeta_{2\lambda}$ in terms of the eigenvalues of the linearized operator for each $\lambda=0,1,2,\dots$, namely
\begin{align*}
	& \psi_{2\lambda+1}(\tau,x)=\sum_{i=0}^{\infty} f_{2\lambda+1}^{(i)}(\tau) \frac{e^{\prime}_{i}(x)}{\omega_{i}}, \quad \sigma_{2\lambda+1}(\tau,x)=\sum_{i=0}^{\infty} g_{2\lambda+1}^{(i)} (\tau) e_{i}(x), \\
	& \xi_{2\lambda}(\tau,x)=\sum_{i=0}^{\infty} p_{2\lambda}^{(i)}(\tau) e_{i}(x), \quad \zeta_{2\lambda}(\tau,x)=\sum_{i=0}^{\infty} q_{2\lambda}^{(i)} (\tau) e_{i}(x),
\end{align*}
substitute these expressions into the recurrence relations above, take the inner product $(\cdot|e^{\prime}_{m})$ for the first equation, the $(\cdot|e_{m})$ for all the other equations and use their orthogonality properties $(e^{\prime}_{n}|e^{\prime}_{m})=\omega_{n}^2 \delta_{nm}$, $(e_{n},e_{m})=\delta_{nm}$, Lemma 2.2 \cite{SecondPublication}. 
Using the notation
\begin{align*}
	\dt{f}(\tau) = \frac{d f(\tau)}{d \tau},\quad 
\ddt{f} (\tau) = \frac{d^2 f(\tau)}{d \tau^2},
\end{align*}
we find
\begin{align*}
	& \omega_{\gamma} \dt{f}_{2\lambda+1}^{(m)} (\tau) = \omega_{m} g_{2\lambda+1}^{(m)} (\tau) +\sum_{\substack{\mu,\nu=0 \\ \mu+\nu=\lambda \\ (\mu,\nu) \neq (\lambda,0) }}^{\lambda}
	\left(
	-\omega_{\gamma,2\nu} \dt{f}_{2\mu+1}^{(m)}(\tau)+ \omega_{m}
	\sum_{i,j=0}^{\infty} C_{ij}^{(m)} p_{2\nu}^{(i)}(\tau) g_{2\mu+1}^{(j)}(\tau)
	\right), \\
	&
	\omega_{\gamma} \dt{g}_{2\lambda+1}^{(m)} (\tau) = - \omega_{m} f_{2\lambda+1}^{(m)} (\tau) -  \sum_{\substack{\mu,\nu=0 \\ \mu+\nu=\lambda \\ (\mu,\nu) \neq (\lambda,0) }}^{\lambda}
	\left(
	\omega_{\gamma,2\nu}  \dt{g}_{2\mu+1}^{(m)}(\tau) + \omega_{m} \sum_{i,j=0}^{\infty}   \overline{C}_{ij}^{(m)} p_{2\nu}^{(i)}(\tau) f_{2\mu+1}^{(j)}(\tau) \right),  \\
	& p_{2(\lambda+1)}^{(m)}(\tau) = 	\sum_{\substack{\rho,k,\nu=0 \\ \rho+k+\nu=\lambda  }}^{\lambda} \sum_{i,j,l=0}^{\infty}
	\Bigg(
	\widetilde{A}_{ijl}^{(m)} f_{2\rho+1}^{(i)}(\tau)f_{2k+1}^{(j)}(\tau)+\widetilde{B}_{ijl}^{(m)} g_{2\rho+1}^{(i)}(\tau)g_{2k+1}^{(j)}(\tau)
	\Bigg) q_{2\nu}^{(l)}(\tau),\nonumber \\
	& q_{2(\lambda+1)}^{(m)}(\tau) = 
	\sum_{\substack{\rho,k,\nu=0 \\ \rho+k+\nu=\lambda  }}^{\lambda} \sum_{i,j,l=0}^{\infty}
	\Bigg(
	\frac{\overline{A}_{ijl}^{(m)}}{\omega_{m}}  f_{2\rho+1}^{(i)}(\tau)f_{2k+1}^{(j)}(\tau)+
	\frac{\overline{B}_{ijl}^{(m)}}{\omega_{m}} g_{2\rho+1}^{(i)}(\tau)g_{2k+1}^{(j)}(\tau)
	\Bigg) q_{2\nu}^{(l)}(\tau),\nonumber
\end{align*} 
where all the interactions with respect to the spatial variable $x \in \left[0,\frac{\pi}{2} \right]$ are included into the following Fourier constants
\begin{align*}
	C_{ij}^{(m)} & :=  \int_{0}^{\frac{\pi}{2}}  e_{i}(x) e_{j}(x) e_{m}(x) \tan^2(x) dx, \\
    \overline{C}_{ij}^{(m)} &:= 
     \int_{0}^{\frac{\pi}{2}}  e_{i}(x) \frac{e_{j}^{\prime}(x)}{\omega_{j}} \frac{e_{m}^{\prime}(x)}{\omega_{m}} \tan^2(x) dx, \\
    \overline{A}_{ijl}^{(m)} &:= 
     \int_{0}^{\frac{\pi}{2}}  \frac{ e_{i}^{\prime}(x)}{\omega_{i}} \frac{ e_{j}^{\prime}(x)}{\omega_{j}} e_{l}(x) \frac{ e_{m}^{\prime}(x)}{\omega_{m}}  \frac{\sin^3(x)}{\cos(x)} dx, \\
    \overline{B}_{ijl}^{(m)} &:= 
     \int_{0}^{\frac{\pi}{2}}  e_{i}(x) e_{j}(x) e_{l}(x) \frac{ e_{m}^{\prime}(x)}{\omega_{m}}  \frac{\sin^3(x)}{\cos(x)} dx, \\
     \widetilde{A}_{ijl}^{(m)} &:= \frac{\overline{A}_{ijl}^{(m)}}{\omega_{m}}  - \int_{0}^{\frac{\pi}{2}} \frac{ e_{i}^{\prime}(x)}{\omega_{i}} \frac{e_{j}^{\prime}(x)}{\omega_{j}} e_{l}(x) 
     \left( 
     \int_{x}^{\frac{\pi}{2}} e_{m}(y) \sin(y)\cos(y) dy
     \right) \tan^2(x)dx, \\
     \widetilde{B}_{ijl}^{(m)} &:= \frac{\overline{B}_{ijl}^{(m)}}{\omega_{m}}  - \int_{0}^{\frac{\pi}{2}} e_{i}(x) e_{j}(x) e_{l}(x) 
     \left( 
     \int_{x}^{\frac{\pi}{2}} e_{m}(y) \sin(y)\cos(y) dy
     \right) \tan^2(x)dx. \\
\end{align*}
The asymptotic behaviour of these constants for large values of $i,j,l$ and $m$ is established in \cite{SecondPublication}.
We also find 
\begin{align*}
	f_{1}^{(m)} (\tau) =\omega_{\gamma} \cos(\tau) \delta_{\gamma}^{m}, \quad 
	g_{1}^{(m)} (\tau) = - \omega_{\gamma} \sin(\tau) \delta_{\gamma}^{m}, \quad
	p_{0}^{(m)} (\tau) = q_{0}^{(m)} (\tau) = (1|e_{m}).
\end{align*}
In Lemma 2.4 \cite{SecondPublication}, we compute
	\begin{align*}
(1|e_{m}):=	\int_{0}^{\frac{\pi}{2}} e_{m}(x) \tan^2(x) dx = \frac{2}{\sqrt{\pi}} \frac{(-1)^m}{\omega_{m}} \sqrt{\omega_{m}^2-1},
\end{align*}
for all $m=0,1,2,\dots$. In addition, we differentiate the first equation with respect to $\tau$ and use the second to obtain the harmonic oscillator equation
\begin{align} \label{HarmonicOscilator}
	\ddt{f}_{2\lambda+1}^{(m)} (\tau) + \left(\frac{\omega_{m}}{\omega_{\gamma}} \right)^2 f_{2\lambda+1}^{(m)} (\tau) = S_{2\lambda+1}^{(m)}(\tau)
\end{align}
where the source term is given by
\begin{align*}
	S_{2\lambda+1}^{(m)}(\tau) &:=\frac{ \omega_{m} }{ \omega_{\gamma}} \sum_{\substack{\mu,\nu=0 \\ \mu+\nu=\lambda \\ (\mu,\nu) \neq (\lambda,0) }}^{\lambda}
	\Bigg[
	-\frac{\omega_{\gamma,2\nu}}{\omega_{\gamma}} 
	\left(
	\dt{g}_{2\mu+1}^{(m)}(\tau) +\frac{ \omega_{\gamma} }{ \omega_{m}} \ddt{f}_{2\mu+1}^{(m)}(\tau)
	\right) \\
	& + 
	 \sum_{i,j=0}^{\infty} 
\left(
	C_{ij}^{(m)} 
	\frac{d}{d \tau} \left(p_{2\nu}^{(i)}(\tau) g_{2\mu+1}^{(j)}(\tau) \right)
-  \frac{\omega_{m}}{\omega_{\gamma}}  \overline{C}_{ij}^{(m)} p_{2\nu}^{(i)}(\tau) f_{2\mu+1}^{(j)}(\tau) 
\right) \Bigg].
\end{align*}	
Finally, we make use of the variation constants formula to solve  \eqref{HarmonicOscilator} and find
\begin{align*}
	f_{2\lambda+1}^{(m)} (\tau) &=
	f_{2\lambda+1}^{(m)} (0) \cos \left( \frac{\omega_{m}}{\omega_{\gamma}} \tau \right) + \frac{\omega_{\gamma}}{\omega_{m}} \dt{f}_{2\lambda+1}^{(m)} (0) \sin \left( \frac{\omega_{m}}{\omega_{\gamma}}\tau \right) + \frac{\omega_{\gamma}}{\omega_{m}}  \int_{0}^{\tau} \sin \left( \frac{\omega_{m}}{\omega_{\gamma}} (\tau-s) \right) S_{2\lambda+1}^{(m)}(s) ds.
\end{align*}
In conclusion, we get for all $m=0,1,2,\dots$ the following recurrence relations. For all $\lambda=1,2,3,\dots$,
\begin{align*}
f_{1}^{(m)} (\tau) &=\omega_{\gamma} \cos(\tau) \delta_{\gamma}^{m}, \\
f_{2\lambda+1}^{(m)} (\tau) & = f_{2\lambda+1}^{(m)} (0) \cos \left( \frac{\omega_{m}}{\omega_{\gamma}} \tau \right) + \frac{\omega_{\gamma}}{\omega_{m}} \dt{f}_{2\lambda+1}^{(m)} (0) \sin \left( \frac{\omega_{m}}{\omega_{\gamma}}\tau \right) + \frac{\omega_{\gamma}}{\omega_{m}}  \int_{0}^{\tau} \sin \left( \frac{\omega_{m}}{\omega_{\gamma}} (\tau-s) \right) S_{2\lambda+1}^{(m)}(s) ds, \\ \\
S_{2\lambda+1}^{(m)}(\tau) &=\frac{ \omega_{m} }{ \omega_{\gamma}} \sum_{\substack{\mu,\nu=0 \\ \mu+\nu=\lambda \\ (\mu,\nu) \neq (\lambda,0) }}^{\lambda}
	\Bigg[
	-\frac{\omega_{\gamma,2\nu}}{\omega_{\gamma}} 
	\left(
	\dt{g}_{2\mu+1}^{(m)}(\tau) +\frac{ \omega_{\gamma} }{ \omega_{m}} \ddt{f}_{2\mu+1}^{(m)}(\tau)
	\right) \\
	& + 
	 \sum_{i,j=0}^{\infty} 
\left(
	C_{ij}^{(m)} 
	\frac{d}{d \tau} \left(p_{2\nu}^{(i)}(\tau) g_{2\mu+1}^{(j)}(\tau) \right)
-  \frac{\omega_{m}}{\omega_{\gamma}}  \overline{C}_{ij}^{(m)} p_{2\nu}^{(i)}(\tau) f_{2\mu+1}^{(j)}(\tau) 
\right) \Bigg]
\\ \\
		g_{1}^{(m)} (\tau) & = - \omega_{\gamma} \sin(\tau) \delta_{\gamma}^{m}, \\
		g_{2\lambda+1}^{(m)} (\tau) & =
	\frac{\omega_{\gamma}}{\omega_{m}} \dt{f}_{2\lambda+1}^{(m)} (\tau)
	+\sum_{\substack{\mu,\nu=0 \\ \mu+\nu=\lambda \\ (\mu,\nu) \neq (\lambda,0)}}^{\lambda}
	\Bigg[
	\frac{ \omega_{\gamma,2\nu}}{\omega_{m}} 
	\dt{f}_{2\mu+1}^{(m)}(\tau) - \sum_{i,j=0}^{\infty} C_{ij}^{(m)} p_{2\nu}^{(i)}(\tau) g_{2\mu+1}^{(j)}(\tau)\Bigg], \\ \\
    p_{0}^{(m)}(\tau) & = \frac{2}{\sqrt{\pi}} \frac{(-1)^m}{\omega_{m}} \sqrt{\omega_{m}^2-1},  \\
    p_{2(\lambda+1)}^{(m)}(\tau) &= 	\sum_{\substack{\rho,k,\nu=0 \\ \rho+k+\nu=\lambda \\  }}^{\lambda} \sum_{i,j,l=0}^{\infty}
		\Bigg[
	\widetilde{A}_{ijl}^{(m)}
	f_{2\rho+1}^{(i)}(\tau)f_{2k+1}^{(j)}(\tau)+
	\widetilde{B}_{ijl}^{(m)}
	g_{2\rho+1}^{(i)}(\tau)g_{2k+1}^{(j)}(\tau)
		\Bigg] q_{2\nu}^{(l)}(\tau), \\ \\
    q_{0}^{(m)}(\tau) & = \frac{2}{\sqrt{\pi}} \frac{(-1)^m}{\omega_{m}} \sqrt{\omega_{m}^2-1}, \\
	q_{2(\lambda+1)}^{(m)}(\tau) &= 
	\sum_{\substack{\rho,k,\nu=0 \\ \rho+k+\nu=\lambda   }}^{\lambda} \sum_{i,j,l=0}^{\infty}
	\Bigg[
	\frac{\overline{A}_{ijl}^{(m)}}{\omega_{m}}  f_{2\rho+1}^{(i)}(\tau)f_{2k+1}^{(j)}(\tau)+
	\frac{\overline{B}_{ijl}^{(m)}}{\omega_{m}} g_{2\rho+1}^{(i)}(\tau)g_{2k+1}^{(j)}(\tau)
		\Bigg] q_{2\nu}^{(l)}(\tau).	
\end{align*}
As pointed out in \cite{13033186}, non-periodic terms appear naturally when the source $S_{2\lambda+1}^{(m)}(\tau)$ has terms of the form $\cos ( \frac{\omega_{m}}{\omega_{\gamma}} \tau )$ or $\sin ( \frac{\omega_{m}}{\omega_{\gamma}} \tau )$ in its Fourier expansion. Indeed, we assume that, for some $\lambda=1,2,3,\dots$, 
\begin{align*}
	S_{2\lambda+1}^{(m)}(\tau) = \sum_{a \in I_{\lambda}} S_{1,2\lambda+1,a}^{(m)} \cos(a x)+\sum_{b \in J_{\lambda}} S_{2,2\lambda+1,b}^{(m)} \sin(b x),
\end{align*}
and in addition there exists an index $m =0,1,2,\dots$ such that
\begin{align*}
	\frac{\omega_{m}}{\omega_{\gamma}}:=a \in I_{\lambda}. 
\end{align*}
Then, the integral
\begin{align*}
	\int_{0}^{\tau} \sin \left( \frac{\omega_{m}}{\omega_{\gamma}} (\tau-s) \right) S_{2\lambda+1}^{(m)}(s) ds
\end{align*}
produces a non-periodic term since
\begin{align*}
	\int_{0}^{\tau} \sin \left( \frac{\omega_{m}}{\omega_{\gamma}} (\tau-s) \right) \cos \left( a s \right) ds & = 
	\int_{0}^{\tau}\sin \left( \frac{\omega_{m}}{\omega_{\gamma}} (\tau-s) \right) \cos \left( \frac{\omega_{m}}{\omega_{\gamma}} s \right) ds = \frac{1}{2} \tau \sin \left( \frac{\omega_{m}}{\omega_{\gamma}} \tau \right).
\end{align*}
Such secular terms are also produced when there exists an $m =0,1,2,\dots$ such that
\begin{align*}
	\frac{\omega_{m}}{\omega_{\gamma}}:=b \in J_{\lambda}. 
\end{align*}
In other words,
\begin{align*}
	\forall \lambda =0,1,2,\dots,~ \exists \text{~a set~} \mathcal{N}_{\lambda}:~\forall m \in  \mathcal{N}_{\lambda}, ~f_{2\lambda+1}^{(m)} \text{~contains non-periodic terms}. 
\end{align*}
Maliborski and Rostworowski \cite{13033186} were able to numerically cancel these secular terms by prescribing the initial data $(f_{2\lambda+1}^{(m)}(0),\dt{f}_{2\lambda+1}^{(m)}(0))$. To explain their approach, we take $f_{1}^{(\gamma)}(0)=1$ and $f_{2\lambda+1}^{(\gamma)}(0)=0$ for $\lambda=1,2,3,\dots$. First, they choose $\dt{f}_{2\lambda+1}^{(m)}(0)=0$ for all $\lambda=0,1,2,\dots$ and all $m=0,1,2,\dots$ to ensure that the source term $ S_{2\lambda+1}^{(m)}(\tau)$ is a series only of cosines. Then, they observed that, for all $\lambda=0,1,2,\dots$, $\gamma \in \mathcal{N}_{\lambda}$ and there is only one secular term in $f_{2\lambda+1}^{(\gamma)}(\tau)$ which can be removed by choosing the frequency 
shift $\omega_{\gamma,2(\lambda-1)}$. Furthermore, for all $m \in \mathcal{N}_{\lambda} \setminus \{\gamma\}$, there are some secular terms which cancel by the structure of the equations, some secular terms cancel by choosing some initial data but some initial data remain free variables at this stage. They choose these free variables together with $\omega_{\gamma,2\lambda}$ to cancel the secular terms in the $f_{2(\lambda+1)+1}$. For more details, see \cite{13033186}. However, there is no proof based on rigorous arguments ensuring that this procedure works for all $\lambda$. For example, we fix $\gamma=0$ and choose
\begin{align*}
	\dt{f}_{2\lambda+1}^{(m)}(0) = 0,\quad \forall \lambda \geq 0, \quad \forall m \geq 0.
\end{align*}
First, we use the recurrence relation above and find periodic expressions for $p_{2}^{(m)}(\tau)$ and $q_{2}^{(m)}(\tau)$ due to the periodicity of $f_{1}^{(m)}(\tau)$ and $g_{1}^{(m)}(\tau)$. Second, we compute
\begin{align*}
	S_{3}^{(m)}(\tau)= A_{3}^{(m)} \cos(\tau)+B_{3}^{(m)} \cos(3\tau),
\end{align*}
for some sequences $\{A_{3}^{(m)}\}_{m=0,1,\dots}$ and $\{B_{3}^{(m)}\}_{m=0,1,\dots}$. Then, the equation
\begin{align*}
	\frac{\omega_{m}}{\omega_{0}}=\frac{3+2m}{3}=1+\frac{2}{3} m \in \{1,3 \}
\end{align*}
has two solutions
\begin{align*}
	m \in \mathcal{N}_{3}:= \{0,3\}.	
\end{align*}
Based on the discussion above, we get two secular terms in the list $\{f_{3}^{(m)}(\tau): m=0,1,2,\dots\}$, one for $m=0$ and one for $m=3$. We see that the secular term for $m=0$ can be canceled by choosing the frequency shift $\theta_{2}$ whereas the secular term for $m=3$ cancels by the structure of the equations meaning $B_{3}^{(3)}=0$ without any choice of the initial data $\{f_{3}^{(m)}(0): m=0,1,2,\dots\}$. Hence, all these are free variables at this stage (meaning for $\lambda=3$) and will be chosen later to cancel all the secular terms in $f_{5}^{(m)}(\tau)$. Specifically, we get
\begin{align*}
		& f_{3}^{(0)}(\tau)=
		\left( \frac{765}{128 \pi}+f_{3}^{(0)}(0) \right) \cos(\tau) 
	- \frac{765}{128 \pi} \cos(3\tau) 
	+ \left(\theta_{2} - \frac{153}{4 \pi} \right) \tau \sin(\tau), \\
		& f_{3}^{(1)}(\tau)=
		\frac{6183}{256 \pi}\sqrt{3} \cos(\tau) 
	+ \frac{765}{256 \pi}\sqrt{3} \cos(3\tau) 
	+ \left(f_{3}^{(1)}(0) - \frac{1737}{64 \pi}\sqrt{3} \right) \cos \left(\frac{5}{3}\tau \right), \\
	& f_{3}^{(2)}(\tau)=
		\frac{3717}{3200 \pi}\sqrt{\frac{3}{2}} \cos(\tau) 
	+ \frac{441}{128 \pi}\sqrt{\frac{3}{2}} \cos(3\tau) 
	+ \left(f_{3}^{(2)}(0) - \frac{7371}{1600 \pi}\sqrt{\frac{3}{2}} \right) \cos \left(\frac{7}{3}\tau \right), \\
	& f_{3}^{(3)}(\tau)=
		-\frac{14607}{4480 \pi}\sqrt{\frac{1}{10}} \cos(\tau) 
	+ \frac{14607}{4480 \pi}\sqrt{\frac{1}{10}} \cos(3\tau) 
	+ f_{3}^{(3)}(3)  \cos \left(3\tau \right), \\
	& f_{3}^{(4)}(\tau)=
		\frac{9999}{62720 \pi}\sqrt{\frac{3}{5}} \cos(\tau) 
	+ \frac{99}{256 \pi}\sqrt{\frac{3}{5}} \cos(3\tau) 
	+ \left(f_{3}^{(4)}(0) - \frac{17127}{31360 \pi}\sqrt{\frac{3}{5}} \right) \cos \left(\frac{11}{3}\tau \right), \\
	& f_{3}^{(5)}(\tau)=
		-\frac{507}{11200 \pi}\sqrt{\frac{3}{7}} \cos(\tau) 
	+ \left(f_{3}^{(5)}(0) + \frac{507}{11200 \pi}\sqrt{\frac{3}{7}} \right) \cos \left(\frac{13}{3}\tau \right), \\
	& f_{3}^{(6)}(\tau)=
		\frac{31}{896 \pi}\sqrt{\frac{1}{7}} \cos(\tau) 
	+ \left(f_{3}^{(6)}(0) - \frac{31}{896 \pi}\sqrt{\frac{1}{7}} \right) \cos \left(5\tau \right), \\
	& f_{3}^{(7)}(\tau)=
		-\frac{11271}{1724800 \pi} \cos(\tau) 
	+ \left(f_{3}^{(7)}(0) + \frac{11271}{1724800 \pi}\right) \cos \left(\frac{17}{3}\tau \right), \\
	& f_{3}^{(8)}(\tau)=
		\frac{1083}{135520 \pi}\sqrt{\frac{1}{5}} \cos(\tau) 
	+ \left(f_{3}^{(8)}(0) - \frac{1083}{135520 \pi}\sqrt{\frac{1}{5}} \right) \cos \left(\frac{19}{3}\tau \right), \\
	& f_{3}^{(9)}(\tau)=
		-\frac{1421}{91520 \pi}\sqrt{\frac{1}{55}} \cos(\tau)
	+ \left(f_{3}^{(9)}(0) + \frac{1421}{91520 \pi}\sqrt{\frac{1}{55}}\right) \cos \left(7\tau \right).
\end{align*}
To ensure the periodicity of $f_{3}^{(0)}(\tau)$, we choose $\theta_{2}=\frac{153}{4\pi}$. Once all secular terms in $f_{3}^{(m)}(\tau)$ are removed, the periodic expression for $f_{3}^{(m)}(\tau)$ implies a periodic expression also for $g_{3}^{(m)}(\tau)$. 
\section{Preliminaries}\label{Preliminaries}
We consider the initial value problem which consists of the cubic wave equation on the Einstein cylinder
\begin{align}\label{maineq1}
	 -  \partial_{t}^2 f (t,\psi) +  \partial_{\psi}^2 f (t,\psi) = \frac{f^3(t,\psi)}{\sin^2(\psi)}, \quad (t,\psi) \in \mathbb{R} \times \left(0,\pi \right), 
\end{align}
subject to Dirichlet boundary conditions
\begin{align}\label{maineq2}
	 f(t,0) = f(t,\pi) = 0, \quad t \in \mathbb{R}.
\end{align}
\subsection{The linear problem}
The operator which governs the solutions to the linearized equation is given by 
\begin{align*}
	L[f](\psi):= - \partial_{\psi}^2 f(\psi).
\end{align*} 
We define the domain of the linearized operator
\begin{align*}
	\mathcal{D} \left( L \right):=
	\left \{
	f \in H^2[0,\pi]: f(0)=f(\pi)=0 \text{ in the trace sense}
	\right \}
\end{align*}
and consider the eigenvalue problem $L f = \omega^2 f$ subject to the Dirichlet boundary conditions \eqref{maineq2}. One finds the eigenvalues 
\begin{align*}
\omega_{i}^2:=(i+1)^2, \quad i=0,1,2,\dots
\end{align*}
and the eigenfunctions
\begin{align*}
e_{i}(\psi):=\sin(\omega_{i}\psi), \quad i=0,1,2,\dots.
\end{align*}
Clearly, the set of eigenfunctions $\{e_{i}: i=0,1,2,\dots \}$ forms an orthogonal basis for $L^2 [0,\pi]$ with respect to the inner product
\begin{align*}
	(f|g):=\frac{2}{\pi} \int_{0}^{\pi} f(\psi) g(\psi) d 
	\psi.
\end{align*}
\subsection{Recurrence relation}
For the non-linear problem, we expand $f$ around the zero solution $f_{0}=0$ as series of powers of epsilon,
\begin{align}\label{series}
	f(t,\psi) = \sum_{\lambda = 0}^{\infty} f_{\lambda}(\tau,\psi) \epsilon^{\lambda} = f_{0}(\tau,\psi) + f_{1}(\tau,\psi) \epsilon + f_{2}(\tau,\psi) \epsilon^{2} + \dots.  
\end{align}
and substitute this expression into \eqref{maineq1}. First, we solve at the linear level and obtain
\begin{align}\label{asxeto1}
	-  \partial_{\tau}^2 f_{1} (\tau,\psi) + \partial_{\psi}^2 f_{1} (\tau,\psi)= 0. 
\end{align}
To initiate the algorithm, we prescribe $f_{1}$ and assume that it consists of 1-mode, namely from the set of all eigenfunctions $\{e_{i}: i=0,1,2,\dots\}$ we pick one dominant term $e_{\gamma}$. For simplicity, we choose $\gamma = 0$ and set
\begin{align*}
	f_{1}(\tau,\psi) =  \cos(\tau) e_{0}(\psi) 
\end{align*} 
which clearly solves \eqref{asxeto1} since
\begin{align*}
	-  \partial_{\tau}^2 f_{1} (\tau,\psi) + \partial_{\psi}^2 f_{1} (t,\psi) = \cos(\tau)e_{0}(\psi) +  \cos(\tau) e_{0}^{\prime \prime}(\psi) =\cos(\tau)e_{0}(\psi) -  \cos(\tau) e_{0}(\psi)  = 0.
\end{align*}
Here, we also perturb the eigenvalue $\omega_{0}=1$ associated to the dominant term we chose by setting
\begin{align*}
	\tau = \Omega t,\quad \Omega^2= \sum_{\lambda = 0}^{\infty} \theta_{\lambda} \epsilon^{\lambda} = \theta_{0} + \theta_{1} \epsilon + \theta_{2} \epsilon^{2} + \dots,\quad \theta_{0} = \omega_{0}^2 =1.
\end{align*}
Second, we compute
\begin{align*}
    & \partial_{\psi}^2 f (t,\psi)  = 
    - L[f(\tau,\cdot) ](\psi) 
    = - \sum_{\lambda=0}^{\infty} 
     L[f_{\lambda}(\tau,\cdot) ](\psi) 
    \epsilon^{\lambda},\\
	& \partial_{t}^2 f (t,\psi) = \Omega^2 \partial_{\tau}^2 f (\tau,\psi)= 
	\left(
	\sum_{\lambda=0}^{\infty} \theta_{\lambda} \epsilon^{\lambda}
	\right)
	\left(
	\sum_{\lambda=0}^{\infty} \partial_{\tau}^2 f (\tau,\psi) \epsilon^{\lambda}
	\right) = 
	\sum_{\lambda=0}^{\infty} \left(
	\sum_{\substack{\mu,\nu=0 \\ \mu + \nu = \lambda }}^{\lambda} \theta_{\nu} \partial_{\tau}^2 f_{\mu} (\tau,\psi)
	\right) \epsilon^{\lambda},  \\
	& \frac{f^3 (t,\psi)}{\sin^2(\psi)} = \frac{1}{\sin^2(\psi)}
	\left(
	\sum_{\lambda=0}^{\infty} f_{\lambda}(\tau,\psi)\epsilon^{\lambda}
	\right)^3 =
	 \sum_{\lambda=0}^{\infty} \left(
	\sum_{\substack{\mu,\nu,\rho=0 \\ \mu + \nu + \rho= \lambda }}^{\lambda} \frac{f_{\mu}(\tau,\psi)f_{\nu}(\tau,\psi)f_{\rho}(\tau,\psi)}{\sin^2(\psi)}
	\right) \epsilon^{\lambda}.
\end{align*}
Now, \eqref{maineq1} boils down to the recurrence relation
\begin{align*}
	\sum_{\substack{\mu,\nu=0 \\ \mu + \nu = \lambda }}^{\lambda} \theta_{\nu}
	\partial_{\tau}^2 f_{\mu} (\tau,\psi) +L[f_{\lambda}(\tau,\cdot) ](\psi) = - \sum_{\substack{\mu,\nu,\rho=0 \\ \mu + \nu + \rho= \lambda }}^{\lambda} 
	\frac{f_{\mu}(\tau,\psi)f_{\nu}(\tau,\psi)f_{\rho}(\tau,\psi)}{\sin^2(\psi)} 
\end{align*}
for all $\lambda=0,1,2,\dots$, which can be written as
\begin{align*}
	 \partial_{\tau}^2 f_{\lambda} (\tau,\psi)+  L[f_{\lambda}(\tau,\cdot) ](\psi) = - \sum_{\substack{\mu,\nu,\rho=0 \\ \mu + \nu + \rho= \lambda}}^{\lambda} \frac{f_{\mu}(\tau,\psi)f_{\nu}(\tau,\psi)f_{\rho}(\tau,\psi)}{\sin^2(\psi)} 
	   - \sum_{\substack{\mu,\nu=0 \\ \mu + \nu = \lambda \\ (\mu,\nu) \neq (\lambda, 0)}}^{\lambda} \theta_{\nu} \partial_{\tau}^2 f_{\mu} (\tau,\psi),
\end{align*}
due to the choice $\theta_{0}=1$. Notice that the conditions $f_{0}=0$ and $\mu+\nu+\rho = \lambda$ restrict the first sum of the right-hand side even further to $\mu,\nu,\rho \neq 0,\lambda-1,\lambda$. Indeed, 
\begin{align*}
	& \mu = 0 \Longrightarrow f_{\mu}=0 \Longrightarrow  f_{\mu}f_{\nu}f_{\rho} =0, \\
	& \mu = \lambda-1 \Longrightarrow \nu + \rho = 1 \Longrightarrow (\nu,\rho)=(1,0) \text{ or } (\nu,\rho)=(0,1) 
	\Longrightarrow f_{\mu}f_{\nu}f_{\rho} = 0,\\
	& \mu = \lambda \Longrightarrow \nu + \rho = 0 \Longrightarrow \nu = \rho = 0 \Longrightarrow f_{\mu}f_{\nu}f_{\rho} = 0.
\end{align*}
Similar computations hold in the case where $\nu,\rho=0,\lambda-1,\lambda$ due to the symmetry of the expression $f_{\mu}f_{\nu}f_{\rho}$. Hence, we can rewrite the latter as 
\begin{align}\label{prealgorithm}
	 \partial_{\tau}^2 f_{\lambda} (\tau,\psi) +  L[f_{\lambda}(\tau,\cdot) ](\psi) & = - \sum_{\substack{\mu,\nu,\rho=1 \\ \mu + \nu + \rho= \lambda  }}^{\lambda-2} \frac{f_{\mu}(\tau,\psi)f_{\nu}(\tau,\psi)f_{\rho}(\tau,\psi)}{\sin^2(\psi)}\nonumber \\
	 &\quad - \sum_{\substack{\mu,\nu=0 \\ \mu + \nu = \lambda \\ (\mu,\nu) \neq ( \lambda,0)}}^{\lambda} \theta_{\nu} \partial_{\tau}^2 f_{\mu} (\tau,\psi)
\end{align}
for all $\lambda=2,3,4,\dots$. Third, we expand all $f_{\lambda}(\tau,\psi)$, $\lambda = 1,2,3,\dots$ with respect to the eigenfunctions to the linearized operator
\begin{align}\label{expansionofphilambda}
	f_{\lambda}(\tau,\psi) = \sum_{i=0}^{\infty} f_{\lambda}^{(i)} (\tau) e_{i}(\psi)
\end{align} 
and use the fact that the $e_{i}$'s are eigenfunctions of $L$, i.e.
\begin{align*}
	L[e_{i}](\psi) = -e_{i}^{\prime \prime}(\psi) = -(\sin(\omega_{i}\psi))^{\prime \prime} =\omega_{i}^2 \sin(\omega_{i}\psi)=  \omega_{i}^2 e_{i}(\psi),
\end{align*}
to compute the left-hand side
\begin{align*}
	 \partial_{\tau}^2 f_{\lambda} (\tau,\psi)+ L[f_{\lambda}(\tau,\cdot) ](\psi) =\sum_{i=0}^{\infty}  
	 \left(
	 \ddt{f}_{\lambda}^{(i)} (\tau) +
	 \omega_{i}^2 f_{\lambda}^{(i)} (\tau) \right) e_{i}(\psi), 
\end{align*}
using the notation
\begin{align*}
\dt{f}(\tau) = \frac{d f(\tau)}{d \tau},\quad 
\ddt{f} (\tau) = \frac{d^2 f(\tau)}{d \tau^2},
\end{align*}
whereas the right-hand side is written as
\begin{align*}
	-\sum_{i,j,k,=0}^{\infty} \sum_{\substack{\mu,\nu,\rho=1 \\ \mu + \nu + \rho= \lambda }}^{\lambda-2}f_{\mu}^{(i)}(\tau)f_{\nu}^{(j)}(\tau)f_{\rho}^{(k)}(\tau) \frac{e_{i}(\psi)e_{j}(\psi)e_{k}(\psi)}{\sin^2(\psi)} -\sum_{i=0}^{\infty} \sum_{\substack{\mu,\nu=0 \\ \mu + \nu = \lambda \\ (\mu,\nu) \neq ( \lambda,0)}}^{\lambda} \theta_{\nu} \ddt{f}_{\mu}^{(i)} (\tau) e_{i}(\psi).
\end{align*}
Next, we pick any $m=0,1,2,\dots$ and take the inner product $(\cdot|e_{m})$ in both side to obtain a simple harmonic oscillator
\begin{align*}
	\ddt{f}_{\lambda}^{(m)}(\tau) +
	 \omega_{m}^2 f_{\lambda}^{(m)} (\tau) = S_{\lambda}^{(m)}(\tau)
\end{align*}
for all $\lambda=2,3,4,\dots$ and $m=0,1,2,\dots$, where the source term reads
\begin{align*}
	S_{\lambda}^{(m)}(\tau):=-\sum_{i,j,k,=0}^{\infty}C_{ijk}^{(m)} \sum_{\substack{\mu,\nu,\rho=1 \\ \mu + \nu + \rho= \lambda  }}^{\lambda-2}f_{\mu}^{(i)}(\tau)f_{\nu}^{(j)}(\tau)f_{\rho}^{(k)}(\tau) 
	-\sum_{\substack{\mu,\nu=0 \\ \mu + \nu = \lambda \\ (\mu,\nu) \neq ( \lambda,0)}}^{\lambda} \theta_{\nu} 
	\ddt{f}_{\mu}^{(m)}(\tau)
\end{align*}
and the interaction coefficients are defined as
\begin{align}\label{FourierConstants}
	C_{ijk}^{(m)}:= \left(  \frac{e_{i}e_{j}e_{k}}{\sin^2} \Bigg| e_{m} \right) = 
	 \frac{2}{\pi} \int_{0}^{\pi} \frac{e_{i}(\psi)e_{j}(\psi)e_{k}(\psi)e_{m}(\psi)}{\sin^2(\psi)} d\psi,\quad i,j,k,m=0,1,2,\dots.
\end{align}
Now, the variation of constants formula yields the $f_{\lambda}^{(m)} (\tau)$ in terms of the $f_{\sigma}^{(i)}(\tau)$, for all $\sigma=1,2,3,\dots,\lambda-1$ and $i=0,1,2,\dots$,
\begin{align*}
	f_{\lambda}^{(m)}(\tau) & =f_{\lambda}^{(m)}(0) \cos(\omega_{m}\tau)+ \frac{1}{\omega_{m}}
	\dt{f}_{\lambda}^{(m)}(0) \sin(\omega_{m} \tau) \\
	& - \frac{1}{\omega_{m}}\cos(\omega_{m}\tau) \int_{0}^{\tau} \sin(\omega_{m}s) S_{\lambda}^{(m)}(s) ds
	+  \frac{1}{\omega_{m}}\sin(\omega_{m}\tau) \int_{0}^{\tau} \cos(\omega_{m}s) S_{\lambda}^{(m)}(s) ds,
\end{align*}
for all $\lambda=2,3,4,\dots$ and $m=0,1,2,\cdots$. Next, we simplify this recurrence relation. Specifically, we look at the terms of the source $S_{\lambda}^{(m)}(\tau)$ which contain time derivatives and use integration by parts
\begin{align*}
	\int_{0}^{\tau} \cos(\omega_{m}s) \ddt{f}_{\mu}^{(m)}(s) ds & = -
	\dt{f}_{\mu}^{(m)}(0)
	+\dt{f}_{\mu}^{(m)}(\tau)
	\cos(\omega_{m}\tau)
	+\omega_{m}  f_{\mu}^{(m)}(\tau) \sin(\omega_{m}\tau) \\
	& - \omega_{m}^2 \int_{0}^{\tau} \cos(\omega_{m}s) f_{\mu}^{(m)}(s) ds, \\
	\int_{0}^{\tau} \sin(\omega_{m}s) 
	\ddt{f}_{\mu}^{(m)}(s) ds & = \omega_{m} f_{\mu}^{(m)}(0) 
	+ \dt{f}_{\mu}^{(m)}(\tau) \sin(\omega_{m}\tau)
	-\omega_{m}  f_{\mu}^{(m)}(\tau) \cos(\omega_{m}\tau) \\
	& - \omega_{m}^2 \int_{0}^{\tau} \sin(\omega_{m}s) f_{\mu}^{(m)}(s) ds
\end{align*}
to obtain
\begin{align*}
	f_{\lambda}^{(m)}(\tau) & = 
	\cos(\omega_{m}\tau)
	\sum_{\substack{\mu,\nu=0 \\ \mu + \nu= \lambda }}^{\lambda} \theta_{\nu} f_{\mu}^{(m)}(0) +
	\frac{1}{\omega_{m}}\sin(\omega_{m}\tau)
	\sum_{\substack{\mu,\nu=0 \\ \mu + \nu= \lambda \\ (\mu,\nu) \neq (\lambda,0) }}^{\lambda} \theta_{\nu} 
	\dt{f}_{\mu}^{(m)}(0) \\
	& - \frac{1}{\omega_{m}} \sum_{i,j,k,=0}^{\infty}C_{ijk}^{(m)} \sum_{\substack{\mu,\nu,\rho=1 \\ \mu + \nu + \rho= \lambda  }}^{\lambda-2} \int_{0}^{\tau} \sin(\omega_{m}(\tau-s)) f_{\mu}^{(i)}(s)f_{\nu}^{(j)}(s)f_{\rho}^{(k)}(s) ds  \\
	& + \omega_{m} \sum_{\substack{\mu,\nu=0 \\ \mu + \nu = \lambda \\ (\mu,\nu) \neq ( \lambda,0)}}^{\lambda} \theta_{\nu} \int_{0}^{\tau} \sin(\omega_{m}(\tau-s)) f_{\mu}^{(m)}(s)ds - \sum_{\substack{\mu,\nu=0 \\ \mu + \nu= \lambda \\ (\mu,\nu) \neq (\lambda,0) }  }^{\lambda} \theta_{\nu} f_{\mu}^{(m)}(\tau).
\end{align*}
Finally, we rewrite the latter as
\begin{align}
	f_{\lambda}^{(m)}(\tau)  &= 
	f_{\lambda}^{(m)}(0)\cos(\omega_{m}\tau) + \frac{1}{\omega_{m}}\sin(\omega_{m}\tau) 
	\sum_{\substack{\mu,\nu=0 \\ \mu + \nu= \lambda \\ (\mu,\nu) \neq (\lambda,0) }}^{\lambda} \theta_{\nu} 
	\dt{f}_{\mu}^{(m)}(0)\label{telikoreccurence1}  \\ 
	& + \sum_{\substack{\mu,\nu=0 \\ \mu + \nu= \lambda \\ (\mu,\nu) \neq (\lambda,0)}}^{\lambda} \theta_{\nu} 
	\left(
	f_{\mu}^{(m)}(0)\cos(\omega_{m}\tau) - f_{\mu}^{(m)}(\tau)
	\right)\nonumber 
	+ \int_{0}^{\tau} \sin(\omega_{m}(\tau-s))N_{\lambda}^{(m)}(s) ds\label{telikoreccurence2} \\ 
N_{\lambda}^{(m)}(s)&:=
\omega_{m}
\sum_{\substack{\mu,\nu=0 \\ \mu + \nu= \lambda \\ (\mu,\nu) \neq (\lambda,0) }}^{\lambda} \theta_{\nu} f_{\mu}^{(m)}(s)
-\frac{1}{\omega_{m}}\sum_{i,j,k,=0}^{\infty} C_{ijk}^{(m)} \sum_{\substack{\mu,\nu,\rho=1 \\ \mu + \nu + \rho= \lambda  }}^{\lambda-2} f_{\mu}^{(i)}(s)f_{\nu}^{(j)}(s)f_{\rho}^{(k)}(s). 
\end{align}
\subsection{Secular terms}\label{SecularTerms} Our goal is to show that the recurrence relation \eqref{telikoreccurence1}-\eqref{telikoreccurence2} produces time periodic solutions by prescribing the initial data. However, secular terms, i.e. terms which destroy the periodicity, appear naturally when for some $\lambda=2,3,\dots$ there exists $m=0,1,2,\dots$ such that $N_{\lambda}^{(m)}(s)$ has either a $\cos(\omega_{m}s)$ or a $\sin(\omega_{m}s)$ in its Fourier expansion. In such a case, secular terms of the form $\tau \sin(\omega_{m}\tau)$, $ \tau \cos(\omega_{m}\tau)$ appear, since
\begin{align*}
	& \int_{0}^{\tau} \sin(\omega_{m}(\tau-s)) \cos(\omega_{m}s) ds = \frac{1}{2} \tau \sin(\omega_{m} \tau), \\
	& \int_{0}^{\tau} \sin(\omega_{m}(\tau-s)) \sin(\omega_{m}s) ds = - \frac{1}{2} \tau \cos(\omega_{m} \tau)+\frac{1}{2 \omega_{m}} \sin(\omega_{m} \tau).
\end{align*}
\subsection{Iterations}\label{iterations}
To see what we can expect and reveal the critical role played by the initial data to rigorously cancel the secular terms, we assume that
\begin{align}\label{initialvelocity}
	\dt{f}_{\lambda}^{(m)}(0) = 0, \quad \forall \lambda \geq 0, \quad \forall m \geq 0
\end{align}
and use an alternative formula for the interaction coefficients (Lemma \ref{ClosedFormulaC}, Appendix \ref{formulaC}) to calculate the recurrence relation \eqref{telikoreccurence1}--\eqref{telikoreccurence2} for $\lambda=2,3,4,5,6$. Due to the choice for the initial velocity \eqref{initialvelocity}, we get that $f_{\lambda}^{(m)}(\tau)$ is a sum only of cosines. Hence, based on the discussion above (section \ref{SecularTerms}), for all $\lambda \geq 2$ and $m \geq 0$, the term responsible for a secular term in $f_{\lambda}^{(m)}(\tau)$ is the coefficient of $\cos(\omega_{m}\tau)$ in the source term $S_{\lambda}^{(m)}(\tau)$. We find\\
\begin{align*}
	& f_{0}^{(m)}(\tau) = 0,\quad  \forall m \geq 0 \\ \\
	& f_{1}^{(m)}(\tau)  =  
	\begin{cases}
		m=0:& \cos(\tau)\\
		m \geq 1:& 0 
	\end{cases} \\ \\
	& f_{2}^{(m)}(\tau)  =  
	\begin{cases}
		m=0:& f_{2}^{(0)}(0)\cos(\tau)+\frac{1}{2} \theta_{1} \tau \sin(\tau) \\
		m \geq 1:& f_{2}^{(m)}(0)\cos(\omega_{m}\tau)
	\end{cases}  \\ \\
	& \text{Before passing to } f_{3}^{(m)} (\tau)\text{ we chooce } \theta_{1}=0.\\ \\
	& f_{3}^{(m)}(\tau)  =  
	\begin{cases}
		m=0:& \left( - \frac{1}{32} + f_{3}^{(0)}(0) \right)
		\cos(\tau)+ \frac{1}{32} \cos(3\tau) +\left( -\frac{3}{8}+ \frac{1}{2}\theta_{2} \right) \tau \sin(\tau)
		 \\
		m \geq 1:& f_{3}^{(m)}(0)\cos(\omega_{m}\tau)
	\end{cases} \\ \\
	& \text{Before passing to } f_{4}^{(m)}(\tau) \text{ we chooce } \theta_{2}=\frac{3}{4}.\\ \\
	& f_{4}^{(m)}(\tau)  =  
	\begin{cases}
		m=0:& \left(  \frac{3}{32} f_{2}^{(0)}(0)+f_{4}^{(0)}(0) \right)\cos(\tau) - \frac{3}{32} f_{2}^{(0)}(0) \cos(3\tau) 
		 + \left( \frac{3}{2} f_{2}^{(0)}(0) +  \frac{1}{2} \theta_{3} \right)\tau \sin(\tau) \\
		m =1 :& f_{4}^{(1)}(0) \cos(2\tau)+f_{2}^{(1)}(0) \left(
		\frac{3}{16} -\frac{1}{8}\cos(2\tau)-\frac{1}{16}\cos(4\tau) + \frac{9}{8} \tau \sin(2\tau)
		\right) \\
		m =2 :& f_{4}^{(2)}(0) \cos(3\tau)+f_{2}^{(2)}(0) \left(
		\frac{3}{32}\cos(\tau) -\frac{3}{64}\cos(3\tau)-\frac{3}{64}\cos(5\tau) + \frac{11}{8}  \tau \sin(3\tau)
		\right) \\
		m =3 :& f_{4}^{(3)}(0) \cos(4\tau)+f_{2}^{(3)}(0) \left(
		\frac{1}{16}\cos(2\tau) -\frac{1}{40}\cos(4\tau)-\frac{3}{80}\cos(6\tau) + \frac{27}{16}  \tau \sin(4\tau)
		\right) \\
	~\quad 	\vdots 
	\end{cases}\\ \\
	& \text{Before passing to } f_{5}^{(m)} (\tau) \text{ we chooce } \theta_{3} = -3 f_{2}^{(0)}(0) \text{ and } f_{2}^{(m)}(0)=0 \text{ for all } m \geq 1. 
	\\ \\
	& f_{5}^{(m)}(\tau)  =  
	\begin{cases}
		m=0:&  \left(
		\frac{31}{1024} + \frac{3}{32} (f_{2}^{(0)}(0))^2 + \frac{3}{32}f_{3}^{(0)}(0)+f_{4}^{(0)}(0) +f_{5}^{(0)}(0)
		\right)\cos(\tau) \\
		& + \left(-
		\frac{15}{512} - \frac{3}{32} (f_{2}^{(0)}(0))^2 - \frac{3}{32} f_{3}^{(0)}(0)
		\right)\cos(3\tau)  - \frac{1}{1024}\cos(5\tau) \\ 
		& +\left(
		 - \frac{9}{256} - \frac{3}{8} (f_{2}^{(0)}(0))^2 + \frac{3}{2} f_{3}^{(0)}(0) + \frac{1}{2} \theta_{4} \right)\tau \sin(\tau) \\
		m =1 :& f_{4}^{(1)}(0) \cos(2\tau)+f_{3}^{(1)}(0) \left(
		\frac{3}{16} -\frac{1}{8}\cos(2\tau)-\frac{1}{16}\cos(4\tau) + \frac{9}{8} \tau \sin(2\tau)
		\right) \\
		m =2 :& f_{4}^{(2)}(0) \cos(3\tau)+f_{3}^{(2)}(0) \left(
		\frac{3}{32}\cos(\tau) -\frac{3}{64}\cos(3\tau)-\frac{3}{64}\cos(5\tau) + \frac{11}{8}  \tau \sin(3\tau)
		\right) \\
		m =3 :& f_{4}^{(3)}(0) \cos(4\tau)+f_{3}^{(3)}(0) \left(
		\frac{1}{16}\cos(2\tau) -\frac{1}{40}\cos(4\tau)-\frac{3}{80}\cos(6\tau) + \frac{27}{16}  \tau \sin(4\tau)
		\right) \\
	~\quad 	\vdots 
	\end{cases} \\ \\
	& \text{Before passing to } f_{6}^{(m)}(\tau) \text{ we chooce } \theta_{4} =  \frac{9}{128} + \frac{3}{4} (f_{2}^{(0)}(0))^2 - 3 f_{3}^{(0)}(0)   \text{ and } f_{3}^{(m)}(0)=0 \text{ for all } m \geq 1.\\ \\
	& f_{6}^{(m)}(\tau)  =  
	\begin{cases}
		m=0:&  \left(
		-\frac{3}{1024} - \frac{99}{512} f_{2}^{(0)}(0) + \frac{1}{32} (f_{2}^{(0)}(0))^3 +\frac{3}{16}f_{2}^{(0)}(0)f_{3}^{(0)}(0) + \frac{35}{32}f_{4}^{(0)}(0) + f_{5}^{(0)}(0)
		\right)\cos(\tau) \\
		& + \left(
		\frac{99}{512}f_{2}^{(0)}(0) - \frac{1}{32} (f_{2}^{(0)}(0))^3 - \frac{3}{16}f_{2}^{(0)}(0) f_{3}^{(0)}(0) -\frac{3}{32}f_{4}^{(0)}(0)
		\right)\cos(3\tau)  + \frac{3}{1024}\cos(5\tau) \\ 
		& +\left(
		  \frac{3}{16}f_{2}^{(0)}(0) + \frac{3}{4} (f_{2}^{(0)}(0))^3 - \frac{3}{4}f_{2}^{(0)}(0) f_{3}^{(0)}(0) + \frac{3}{2}f_{4}^{(0)}(0) + \frac{1}{2} \theta_{5} \right)\tau \sin(\tau) \\
		m =1 :& f_{5}^{(1)}(0) \cos(2\tau)+f_{4}^{(1)}(0) \left(
		\frac{3}{16} - \frac{1}{8}\cos(2\tau)-\frac{1}{16}\cos(4\tau) + \frac{9}{8} \tau \sin(2\tau)
		\right) \\
		m =2 :& f_{5}^{(2)}(0) \cos(3\tau)+f_{4}^{(2)}(0) \left(
		\frac{3}{32}\cos(\tau) -\frac{3}{64}\cos(3\tau)-\frac{3}{64}\cos(5\tau) + \frac{11}{8}  \tau \sin(3\tau)
		\right) \\
		m =3 :& f_{5}^{(3)}(0) \cos(4\tau)+f_{4}^{(3)}(0) \left(
		\frac{1}{16}\cos(2\tau) -\frac{1}{40}\cos(4\tau)-\frac{3}{80}\cos(6\tau) + \frac{27}{16}  \tau \sin(4\tau)
		\right) \\
	~\quad 	\vdots 
	\end{cases} \\ \\
	& \text{Before passing to } f_{7}^{(m)}(\tau) \text{ we chooce } \theta_{5} = - \frac{3}{8}f_{2}^{(0)}(0) - \frac{3}{2} (f_{2}^{(0)}(0))^3 + \frac{3}{2} f_{2}^{(0)}(0) f_{3}^{(0)}(0) - 3f_{4}^{(0)}(0)  \text{ and } \\
	& f_{4}^{(m)}(0)=0 \text{ for all } m \geq 1.
\end{align*}
We observe that 
\begin{itemize}
 \item For all $\lambda \geq 2$ and $m \geq 0$, before passing to $f_{\lambda+1}^{(m)}(\tau)$, we tune $f_{\lambda}^{(m)}(\tau)$ by choosing the frequency shift $\theta_{\lambda-1}$ as well as the initial data $f_{\lambda-2}^{(m)}(0)$ only for $m \geq 1$ whereas all the $f_{\lambda-2}^{(0)}(0)$ are free variables.
\item For all $\lambda \geq 2$ and $m=0$, there is only one secular term in $f_{\lambda}^{(0)}(\tau)$ which can be removed by prescribing the frequency shift $\theta_{\lambda-1}$. Once this choice is made, $\theta_{\lambda-1}$ depends only on the free initial data $f_{\sigma}^{(0)}(0)$ for $\sigma = 1,2,3,\dots,\lambda-2$. Furthermore, this secular term 
	\begin{align*}
		p_{0}(\psi):=\tau \sin(\tau)
	\end{align*}
	 is the same for all $\lambda \geq 2$ when $m=0$.
	\item For all $\lambda \geq 4$ and $m \geq 1$, there is only one secular term in $f_{\lambda}^{(m)}(\tau)$ which can be removed by prescribing the initial data $f_{\lambda-2}^{(m)}(0)=0$. Furthermore, this secular term 
	\begin{align*}
		p_{m}(\psi):=
		\begin{cases}
			m=1:& \frac{3}{16} - \frac{1}{8}\cos(2\tau)-\frac{1}{16}\cos(4\tau) +\frac{9}{8} \tau \sin(\tau), \\
			m=2:& \frac{3}{32}\cos(\tau)-\frac{3}{64}\cos(3\tau) -\frac{3}{64}\cos(5\tau)+\frac{11}{8} \tau \sin(3\tau)\\ 
			m=3:& \frac{1}{16}\cos(2\tau)-\frac{1}{40}\cos(4\tau) -\frac{3}{80}\cos(6\tau)+\frac{27}{16} \tau \sin(4\tau)\\ 
			~\quad \vdots &
		\end{cases}
	\end{align*}
	 is the same for all $\lambda \geq 4$ when $m \geq 1$.
\item The choice $f_{\lambda}^{(m)}(0)=0$ for all $\lambda=2,3,4,\dots$ and all $m=1,2,3,\dots$ implies $f_{\lambda}^{(m)}(\tau)=0$ for all $\lambda=2,3,4,\dots$, all $m=1,2,3,\dots$ and all $\tau$. 
\end{itemize}
\section{Main results and proofs}\label{Statement}
\subsection{Main results}
In the following, $[x] \in \mathbb{N}\cup\{0\}$ stands for the integer part of $x\in \mathbb{R}$, i.e. the largest integer which is smaller or equal to $x$. We prove the following two results.
%
%
%
%
%
%
\begin{theorem}\label{maintheorem}
	Assume that $\dt{f}_{p}^{(l)}(0) = 0$ for all $p=0,1,2,\dots$ and $l=0,1,2,\dots$. Then, for all $\lambda \geq 2$, there exist constants $c_{\lambda}^{\alpha} $ for $\alpha=0,1,2,\dots,\left[\frac{\lambda-1}{2} \right]$ depending only on $f_{\sigma}^{(0)}(0)$ for $\sigma=1,2,3,\dots,\lambda-2$ such that the following is true. For $m=0$, we have
\begin{align*}
     f_{\lambda}^{(0)}(\tau)=
            f_{\lambda}^{(0)}(0) \cos(\tau) + \sum_{\alpha=0}^{\left[\frac{\lambda-1}{2} \right]} c_{\lambda}^{\alpha} \cos((1+2\alpha)\tau),
\end{align*}	
by choosing $\theta_{0},\theta_{1},\dots,\theta_{\lambda-1}$. For $m \geq 1$, we have
\begin{align*}
     f_{\lambda}^{(m)}(\tau)=
 f_{\lambda}^{(m)}(0) \cos(\omega_{m}\tau), 
\end{align*}
by choosing $f_{\lambda-2}^{(m)}(0) = 0$.
\end{theorem}
	Interestingly, our analysis reveals that there exists a periodic solution to \eqref{maineq11}-\eqref{maineq22} if and only if $f(t,\cdot)$ is proportional to $e_{0}$. Indeed, for all $\lambda=2,3,4,\dots$, the choice of the initial data $f_{\lambda-2}^{(m)}(0) = 0$ for all $m=1,2,3,\dots$, implies $f_{\lambda-2}^{(m)}(\tau)=0$ for all $m=1,2,3,\dots$ and all $\tau$,  whereas $f_{\lambda-2}^{(0)}(0)$ are all free variables. Hence,
	\begin{align*}
		f_{\lambda}(\tau,\psi) & = \sum_{i=0}^{\infty} f_{\lambda}^{(i)} (\tau) e_{i}(\psi) = f_{\lambda}^{(0)} (\tau) e_{0}(\psi)
	\end{align*}
	and $f$ turns out to be proportional to $e_{0}$, 
	\begin{align*}	
		f(t,\psi) & = \sum_{\lambda=0}^{\infty} f_{\lambda} (\tau,\psi) \epsilon^{\lambda} = \left( \sum_{\lambda=0}^{\infty} f_{\lambda}^{(0)} (\tau) \epsilon^{\lambda}
		\right)e_{0}(\psi).
	\end{align*}
	Now, we set
\begin{align*}
	F(\tau):=\sum_{\lambda=0}^{\infty} f_{\lambda}^{(0)} (\tau) \epsilon^{\lambda},\quad \tau = \Omega t, \quad \Omega^2 := \sum_{\lambda=0}^{\infty} \theta_{\lambda} \epsilon^{\lambda}
\end{align*}	
	and plug $f(t,\psi)= F(\tau)e_{0}(\psi)$ into \eqref{maineq11} to obtain 
\begin{align*}
	\Omega^2 \frac{d^2 F(\tau)}{ d\tau^2 } + F(\tau)+F^3(\tau) = 0.
\end{align*}
For $G(t)=F(\Omega t)=F(\tau)$, the latter reads
\begin{align}\label{ODE}
	\frac{d^2 G(t)}{ dt ^2 } + G(t)+ G^3(t) = 0.
\end{align}
Observe that the condition $\dt{f}_{p}^{(l)}(0) = 0$ for all $p=0,1,2,\dots$ and $l=0,1,2,\dots$ implies $G^{\prime}(0)=0$ where here and in the following we write $' = d/dt$.
\begin{theorem} \label{maintheorem2}
	All initial data $(G(0),G^{\prime}(0))$ such that $G(0) \neq 0$ and $G^{\prime}(0)=0$ lead to time periodic solutions to the second order ordinary differential equation \eqref{ODE}.
\end{theorem}
\subsection{Proofs} We begin with the proof of Theorem \ref{maintheorem}.
\begin{proof}[Proof of Theorem \ref{maintheorem}]
	The proof is based on induction in $\lambda$. 
\begin{flushleft}
	\underline{\text{Step 1}: Proof for $\lambda=2,3$.}
\end{flushleft}	
Recall that, for $\lambda=1$, we have
\begin{align*}
	 \partial_{\tau}^2 f_{1} (\tau,\psi)
	 +  L[f_{1}(\tau,\cdot) ](\psi) = 0
\end{align*} 
and we chose
\begin{align*}
	f_{1}(\tau,\psi) = \cos(\tau) e_{0}(\psi)
\end{align*}
to initiate the algorithm which in turn implies
\begin{align*}
	f_{1}^{(m)} (\tau) = \delta_{0}^{m} \cos(\tau),
\end{align*}
for all $m \geq 0$. For small values of $\lambda$, such as $\lambda=2$ and $\lambda=3$, it is more natural to use \eqref{prealgorithm} before passing to \eqref{telikoreccurence1}-\eqref{telikoreccurence2}. For $\lambda=2$, we have
\begin{align*}
	\partial_{\tau}^2 f_{2} (\tau,\psi) +  L[f_{2}(\tau,\cdot) ](\psi) & = - \theta_{1} \partial_{\tau}^2 f_{1} (\tau,\psi) \\
	 & =  \theta_{1} \cos(\tau) e_{0}(\psi) 
\end{align*}
and using the expansion \eqref{expansionofphilambda} we infer
\begin{align*}
	\ddt{f}_{2}^{(m)}(\tau) +
	 \omega_{m}^2 f_{2}^{(m)} (\tau) = \theta_{1}\delta_{0}^{m} \cos(\tau),
\end{align*}
for all $m \geq 0$, which in turn implies
\begin{align*}
	f_{2}^{(m)}(\tau)  =  
	\begin{cases}
		m=0:& f_{2}^{(0)}(0)\cos(\tau)+\frac{1}{2} \theta_{1} \tau \sin(\tau) \\
		m \geq 1:& f_{2}^{(m)}(0)\cos(\omega_{m}\tau), 
	\end{cases}
\end{align*}
where the assumption of the initial velocity of $f_{2}^{(m)}$ was used. To get a periodic solution we are forced to choose $\theta_{1}=0$. Similarly, for $\lambda=3$, we have
\begin{align*}
	\partial_{\tau}^2 f_{3} (\tau,\psi) +  L[f_{3}(\tau,\cdot) ](\psi) & = \frac{f_{1}^3(\tau,\psi)}{\sin^2(\psi)}  - \theta_{1} \partial_{\tau}^2 f_{2} (\tau,\psi)  - \theta_{2} \partial_{\tau}^2 f_{1} (\tau,\psi) \\
	 & = - \cos^3(\tau) e_{0}(\psi) + \theta_{2} \cos(\tau) e_{0}(\psi) \\
	 & = \Big[ \left(\theta_{2} -\frac{3}{4} \right) \cos(\tau) - \frac{1}{4} \cos(3\tau) \Big] e_{0}(\psi)
\end{align*}
since $\theta_{1}=0$ and $f_{1}(\tau,\psi)=\cos(\tau) e_{0}(\psi)$. As before, we make use of the expansion \eqref{expansionofphilambda} to obtain 
\begin{align*}
	\ddt{f}_{3}^{(m)}(\tau) +
	 \omega_{m}^2 f_{3}^{(m)} (\tau) = \Big[ \left(\theta_{2} -\frac{3}{4} \right) \cos(\tau) - \frac{1}{4} \cos(3\tau) \Big] \delta_{0}^{m},
	 \end{align*}
for all $m \geq 0$, which in turn implies
\begin{align*}
	f_{3}^{(m)}(\tau)  =  
	\begin{cases}
		m=0:& f_{3}^{(0)}(0)\cos(\tau) - \frac{1}{32} 
		\cos(\tau)+ \frac{1}{32} \cos(3\tau) +\frac{1}{2}\left( \theta_{2}-\frac{3}{4} \right) \tau \sin(\tau) \\
		m \geq 1:& f_{3}^{(m)}(0)\cos(\omega_{m}\tau),
	\end{cases}
\end{align*}
where again the assumption of the initial velocity of $f_{3}^{(m)}$ was used. To get a periodic solution we are forced to choose $\theta_{2}=\frac{3}{4}$. 
\begin{flushleft}
	\underline{\text{Step 2}: Induction hypothesis.}
\end{flushleft}	
	We assume that there exists some $\lambda \geq 4$ such that the statement holds true for all $\xi=2,3,,\dots,\lambda-1$. In other words, we assume that for all $\xi=2,3,\dots,\lambda-1$, we have already prescribed $\theta_{0},\theta_{1},\dots,\theta_{\xi-1}$ so that
	\begin{align*}
		f_{\xi}^{(0)}(\tau)=
            f_{\xi}^{(0)}(0) \cos(\tau) + \sum_{\alpha=0}^{\left[\frac{\xi-1}{2} \right]} c_{\xi}^{\alpha} \cos((1+2\alpha)\tau),
	\end{align*}
	and have already set $f_{\xi-2}^{(m)}(0) = 0,~ m\geq 1$ so that
	\begin{align*}
	f_{\xi}^{(m)}(\tau)= 	f_{\xi}^{(m)}(0) \cos(\omega_{m}\tau), \quad  m\geq 1.
	\end{align*}	
\begin{flushleft}
	\underline{\text{Step 3}: Proof for $\xi=\lambda$.}
\end{flushleft}	
\begin{flushleft}
	\underline{\text{Step 3A}: Proof for $\xi=\lambda$ and $m=0$.}
\end{flushleft}
Based on the recurrence relation \eqref{telikoreccurence1}-\eqref{telikoreccurence2}, $f_{\lambda}^{(0)}(\tau)$ is given in terms of the $f_{\xi}^{(i)}(\tau)$ with $\xi=0,1,2,\dots,\lambda-1$ and $i=0,1,2,\dots$ by
\begin{align*}
	f_{\lambda}^{(0)}(\tau)  &= 
	f_{\lambda}^{(0)}(0)\cos(\tau) + \sum_{\substack{\mu,\nu=0 \\ \mu + \nu= \lambda \\ (\mu,\nu) \neq (\lambda,0)}}^{\lambda} \theta_{\nu} 
	\left(
	f_{\mu}^{(0)}(0)\cos(\tau) - f_{\mu}^{(0)}(\tau)
	\right)
	+ \int_{0}^{\tau} \sin(\tau-s)N_{\lambda}^{(0)}(s) ds, \\ 
N_{\lambda}^{(0)}(s)&:=
\sum_{\substack{\mu,\nu=0 \\ \mu + \nu= \lambda \\ (\mu,\nu) \neq (\lambda,0) }}^{\lambda} \theta_{\nu} f_{\mu}^{(0)}(s)
-\sum_{i,j,k,=0}^{\infty} C_{ijk}^{(0)} \sum_{\substack{\mu,\nu,\rho=1 \\ \mu + \nu + \rho= \lambda  }}^{\lambda-2} f_{\mu}^{(i)}(s)f_{\nu}^{(j)}(s)f_{\rho}^{(k)}(s). 
\end{align*}
First, we compute
\begin{align*}
	\sum_{\substack{\mu,\nu=0 \\ \mu + \nu= \lambda \\ (\mu,\nu) \neq (\lambda,0)}}^{\lambda} \theta_{\nu} 
	\left(
	f_{\mu}^{(0)}(0)\cos(\tau) - f_{\mu}^{(0)}(\tau)
	\right).
\end{align*}
Notice that this quantity depends only on $f_{\xi}^{(0)}(\tau)$, $\xi =0,1,2,\dots, \lambda -1$ and so 
\begin{align*}
	& \sum_{\substack{\mu,\nu=0 \\ \mu + \nu= \lambda \\ (\mu,\nu) \neq (\lambda,0)}}^{\lambda} \theta_{\nu} 
	\left(
	f_{\mu}^{(0)}(0)\cos(\tau) - f_{\mu}^{(0)}(\tau)
	\right)  = \\  
	& \sum_{\substack{\mu,\nu=0 \\ \mu + \nu= \lambda \\ (\mu,\nu) \neq (\lambda,0)}}^{\lambda} \theta_{\nu} 
	\left(
	f_{\mu}^{(0)}(0)\cos(\tau) -  f_{\mu}^{(0)}(0) \cos(\tau) - \sum_{\alpha=0}^{\left[\frac{\mu-1}{2} \right]} c_{\mu}^{\alpha} \cos((1+2\alpha)\tau) \right) = \\
	&- \sum_{\substack{\mu,\nu=0 \\ \mu + \nu= \lambda \\ (\mu,\nu) \neq (\lambda,0)}}^{\lambda} \theta_{\nu} 
	 \sum_{\alpha=0}^{\left[\frac{\mu-1}{2} \right]} c_{\mu}^{\alpha} \cos((1+2\alpha)\tau) =
	 \sum_{\alpha=0}^{\left[\frac{\lambda-1}{2} \right]} \Bigg[ -
	  \sum_{  \substack{  \mu,\nu=0 \\ \mu + \nu= \lambda \\ (\mu,\nu) \neq (\lambda,0)\\ 0 \leq \alpha \leq \left[\frac{\mu-1}{2}  \right] }  }^{\lambda} \theta_{\nu}  c_{\mu}^{\alpha} 
	  \Bigg] \cos((1+2\alpha)\tau).
\end{align*}
Second, we compute the first part of of the integral of $\sin(\tau-s)N_{\lambda}^{(0)}(s)$ namely
\begin{align*}
	\sum_{\substack{\mu,\nu=0 \\ \mu + \nu= \lambda \\ (\mu,\nu) \neq (\lambda,0) }}^{\lambda} \theta_{\nu} \int_{0}^{\tau} \sin(\tau-s) f_{\mu}^{(0)}(s) ds.
\end{align*}
For simplicity, we define
\begin{align*}
	d_{\mu}^{\alpha}:=
	\begin{cases}
		\alpha=0:& f_{\mu}^{(0)}(0)+ c_{\mu}^{0} \\
		\alpha \geq 1:& c_{\mu}^{\alpha}
	\end{cases}
\end{align*}
and $f_{\mu}^{(0)}(\tau)$ can be written in a more compact form as
\begin{align*}
		f_{\mu}^{(0)}(\tau)=
            f_{\mu}^{(0)}(0) \cos(\tau) + \sum_{\alpha=0}^{\left[\frac{\mu-1}{2} \right]} c_{\mu}^{\alpha} \cos((1+2\alpha)\tau) = \sum_{\alpha=0}^{\left[\frac{\mu-1}{2} \right]}  d_{\mu}^{\alpha} \cos((1+2\alpha)\tau).
\end{align*}
As before, notice that this quantity depends only on $f_{\xi}^{(0)}(\tau)$, $\xi =0,1,2,\dots, \lambda -1$ and so
\begin{align*}
	& \sum_{\substack{\mu,\nu=0 \\ \mu + \nu= \lambda \\ (\mu,\nu) \neq (\lambda,0) }}^{\lambda} \theta_{\nu} \int_{0}^{\tau} \sin(\tau-s) f_{\mu}^{(0)}(s) ds =\\
	 &\sum_{\substack{\mu,\nu=0 \\ \mu + \nu= \lambda \\ (\mu,\nu) \neq (\lambda,0) }}^{\lambda} \theta_{\nu} 
	\Bigg[
	d_{\mu}^{0}  \int_{0}^{\tau}  \sin(\tau-s)  \cos(s) ds + \sum_{\alpha=1}^{\left[\frac{\mu-1}{2} \right]} d_{\mu}^{\alpha} \int_{0}^{\tau}  \sin(\tau-s) \cos((1+2\alpha)s) ds \Bigg]=\\
	 &\sum_{\substack{\mu,\nu=0 \\ \mu + \nu= \lambda \\ (\mu,\nu) \neq (\lambda,0) }}^{\lambda} \theta_{\nu} 
	\Bigg[
	  \frac{d_{\mu}^{0}}{2} \tau \sin(\tau) + \sum_{\alpha=1}^{\left[\frac{\mu-1}{2} \right]} d_{\mu}^{\alpha} \left( \frac{\cos(\tau)}{4 \alpha (1+\alpha)} - \frac{\cos((1+2\alpha)\tau)}{4 \alpha (1+\alpha)} \right) \Bigg]= \\
	 & \Bigg[ \sum_{\substack{\mu,\nu=0 \\ \mu + \nu= \lambda \\ (\mu,\nu) \neq (\lambda,0) }}^{\lambda} \frac{1}{2}\theta_{\nu} 
	 d_{\mu}^{0}   \Bigg] \tau \sin(\tau)
	 + \Bigg[ \sum_{\substack{\mu,\nu=0 \\ \mu + \nu= \lambda \\ (\mu,\nu) \neq (\lambda,0) }}^{\lambda} \theta_{\nu} \sum_{\alpha=1}^{\left[\frac{\mu-1}{2} \right]} \frac{d_{\mu}^{\alpha} }{4 \alpha (1+\alpha)} \Bigg] \cos(\tau) \\
	 &- \sum_{\substack{\mu,\nu=0 \\ \mu + \nu= \lambda \\ (\mu,\nu) \neq (\lambda,0) }}^{\lambda} \theta_{\nu} \sum_{\alpha=1}^{\left[\frac{\mu-1}{2} \right]} \frac{d_{\mu}^{\alpha}}{4 \alpha (1+\alpha)} \cos((1+2\alpha)\tau)
\end{align*}
which is of the form
\begin{align*}
	\Bigg[ \sum_{\substack{\mu,\nu=0 \\ \mu + \nu= \lambda \\ (\mu,\nu) \neq (\lambda,0) }}^{\lambda} \frac{1 }{2}\theta_{\nu}
	 d_{\mu}^{0}   \Bigg] \tau \sin(\tau) + \mathcal{M}_{\lambda} \cos(\tau) + \sum_{\beta=0}^{\left[\frac{\lambda-1}{2} \right]} \mathcal{N}_{\lambda}^{\beta} \cos{((1+2\beta)\tau}),
\end{align*}
where $\mathcal{M}_{\lambda}$ and $\mathcal{N}_{\lambda}^{\beta}$ depend only on the $c_{\mu}^{\alpha}$. Third, we focus on the second part of the integral of $\sin(\tau-s)N_{\lambda}^{(0)}(s)$ namely on
\begin{align}\label{symmquant}
	  -\sum_{i,j,k,=0}^{\infty} C_{ijk}^{(0)} \sum_{\substack{\mu,\nu,\rho=1 \\ \mu + \nu + \rho= \lambda  }}^{\lambda-2} \int_{0}^{\tau} \sin(\tau-s)f_{\mu}^{(i)}(s)f_{\nu}^{(j)}(s)f_{\rho}^{(k)}(s) ds.
\end{align}
Now, this quantity depends only on $f_{\xi}^{(l)}(\tau)$, $\xi =1,2,\dots, \lambda -2$ and $l =0,1,2,\dots$. Since $l=0$ is also included, we split the latter into the following cases.
\begin{flushleft}
	\underline{Case 1: $i,j,k \neq 0$}
\end{flushleft}
Then, 
\begin{align*}
	& -\sum_{\substack{  i,j,k,=0 \\ i,j,k \neq 0 }}^{\infty} C_{ijk}^{(0)} \sum_{\substack{\mu,\nu,\rho=1 \\ \mu + \nu + \rho= \lambda  }}^{\lambda-2} \int_{0}^{\tau} \sin(\tau-s)f_{\mu}^{(i)}(s)f_{\nu}^{(j)}(s)f_{\rho}^{(k)}(s) ds = \\
	& -\sum_{\substack{  i,j,k=1 }}^{\infty} C_{ijk}^{(0)} \sum_{\substack{\mu,\nu,\rho=1 \\ \mu + \nu + \rho= \lambda  }}^{\lambda-2}f_{\mu}^{(i)}(0)f_{\nu}^{(j)}(0)f_{\rho}^{(k)}(0) \int_{0}^{\tau} \sin(\tau-s)\cos(\omega_{i}s)\cos(\omega_{j}s)\cos(\omega_{k}s) ds.
\end{align*}
However, we assumed that $f_{\xi-2}^{(m)}(0) = 0$ for all $m =1,2,3,\dots$ and $\xi=2,3,\dots,\lambda-1$ which is equivalent to $f_{\xi}^{(m)}(0) = 0$ for all $m =1,2,3,\dots$ and $\xi=0,1,\dots,\lambda-3$. So, the only case where this sum is not zero due to the choice of the initial data is when $\mu = \nu = \rho = \lambda -2$. Then, $\lambda = \mu + \nu + \rho = 3( \lambda -2)$ which holds only when $\lambda=3$ and we initially assumed that $\lambda \geq 4$. In conclusion, this sum equals to zero.
\begin{flushleft}
	\underline{Case 2: ($i=0$ and $j,k \neq 0$) or ($j=0$ and $i,k \neq 0$) or ($k=0$ and $i,j \neq 0$)}
\end{flushleft}
Notice that the quantity inside the sum in \eqref{symmquant} is symmetric with respect to $i,j,k$. For example, 
\begin{align*}
	& C_{ijk}^{(0)} \sum_{\substack{\mu,\nu,\rho=1 \\ \mu + \nu + \rho= \lambda  }}^{\lambda-2} \int_{0}^{\tau} \sin(\tau-s)f_{\mu}^{(i)}(s)f_{\nu}^{(j)}(s)f_{\rho}^{(k)}(s) ds = \\
	&C_{jik}^{(0)} \sum_{\substack{\mu,\nu,\rho=1 \\ \mu + \nu + \rho= \lambda  }}^{\lambda-2} \int_{0}^{\tau} \sin(\tau-s)f_{\mu}^{(j)}(s)f_{\nu}^{(i)}(s)f_{\rho}^{(k)}(s) ds = \\
	&C_{jik}^{(0)} \sum_{\substack{\mu,\nu,\rho=1 \\ \mu + \nu + \rho= \lambda  }}^{\lambda-2} \int_{0}^{\tau} \sin(\tau-s)f_{\nu}^{(i)}(s) f_{\mu}^{(j)}(s)f_{\rho}^{(k)}(s) ds = \\
	&C_{jik}^{(0)} \sum_{\substack{\nu,\mu,\rho=1 \\  \nu +\mu+ \rho= \lambda  }}^{\lambda-2} \int_{0}^{\tau} \sin(\tau-s)f_{\mu}^{(i)}(s) f_{\nu}^{(j)}(s)f_{\rho}^{(k)}(s) ds.
\end{align*}
Due to this symmetry, we only consider one of these cases. In particular, we assume that $i=0$ and $j,k \neq 0$. Then, we compute 
\begin{align*}
C_{0jk}^{(0)}=\frac{2}{\pi} \int_{0}^{\pi} \frac{\sin^2(\omega_{0}\psi)\sin(\omega_{j}\psi)\sin(\omega_{k}\psi)}{\sin^2(\psi)}  d \psi =\frac{2}{\pi} \int_{0}^{\pi} \sin(\omega_{j}\psi)\sin(\omega_{k}\psi) d \psi =  \delta_{j}^{k}	
\end{align*}
and
\begin{align*}
	& -\sum_{\substack{  j,k,=0 \\ j,k \neq 0 }}^{\infty} C_{0jk}^{(0)} \sum_{\substack{\mu,\nu,\rho=1 \\ \mu + \nu + \rho= \lambda  }}^{\lambda-2} \int_{0}^{\tau} \sin(\tau-s)f_{\mu}^{(0)}(s)f_{\nu}^{(j)}(s)f_{\rho}^{(k)}(s) ds = \\
	& -\sum_{\substack{  k=1 }}^{\infty}  \sum_{\substack{\mu,\nu,\rho=1 \\ \mu + \nu + \rho= \lambda  }}^{\lambda-2} \int_{0}^{\tau} \sin(\tau-s)f_{\mu}^{(0)}(s)f_{\nu}^{(k)}(s)f_{\rho}^{(k)}(s) ds = \\
	& -\sum_{\substack{  k=1 }}^{\infty}  \sum_{\substack{\mu,\nu,\rho=1 \\ \mu + \nu + \rho= \lambda  }}^{\lambda-2} \int_{0}^{\tau} \sin(\tau-s) \Bigg[
	f_{\mu}^{(0)}(0) \cos(s) + \sum_{\alpha=0}^{\left[\frac{\mu-1}{2} \right]} c_{\mu}^{\alpha} \cos((1+2\alpha)s) \Bigg] 
	f_{\nu}^{(k)}(0)f_{\rho}^{(k)}(0)\cos^2(\omega_{k}s) 
	ds.
\end{align*}
Recall that $f_{\mu}^{(0)}(0)$ are all free variables. As before, we assumed that $f_{\xi-2}^{(m)}(0) = 0$ for all $m =1,2,3,\dots$ and $\xi=2,3,\dots,\lambda-1$ which is equivalent to $f_{\xi}^{(m)}(0) = 0$ for all $m =1,2,3,\dots$ and $\xi=0,1,\dots,\lambda-3$. So, the only case where this sum is not zero due to the choice of the initial data is when $\nu = \rho = \lambda -2$. Then, $\lambda = \mu + \nu + \rho = \mu + 2(\lambda -2)$  which implies $\mu+\lambda=4$ and we initially assumed that $\lambda \geq 4$. Hence, the only possible case is $\mu=0$ only for $\lambda =4$ which contradicts the condition $\mu \geq 1$. In conclusion, the sum equals to zero.
\begin{flushleft}
	\underline{Case 3: ($i \neq 0$ and $j,k =0$) or ($j \neq 0$ and $i,k = 0$) or ($k \neq 0$ and $i,j = 0$)}
\end{flushleft}
Again, due to the symmetry, we only consider one of these cases. In particular, we assume that $i \neq 0$ and $j,k = 0$. Then, we have 
\begin{align*}
	& -\sum_{\substack{  i,j,k=0 \\ i\neq 0, j,k=0}}^{\infty} C_{ijk}^{(0)} \sum_{\substack{\mu,\nu,\rho=1 \\ \mu + \nu + \rho= \lambda  }}^{\lambda-2} \int_{0}^{\tau} \sin(\tau-s)f_{\mu}^{(i)}(s)f_{\nu}^{(j)}(s)f_{\rho}^{(k)}(s) ds = \\
	& -\sum_{\substack{  i=1 }}^{\infty} C_{i00}^{(0)} \sum_{\substack{\mu,\nu,\rho=1 \\ \mu + \nu + \rho= \lambda  }}^{\lambda-2} \int_{0}^{\tau} \sin(\tau-s)f_{\mu}^{(i)}(s)f_{\nu}^{(0)}(s)f_{\rho}^{(0)}(s) ds
\end{align*}
and since
\begin{align*}
C_{i00}^{(0)}=\frac{2}{\pi} \int_{0}^{\pi} \frac{\sin^3(\omega_{0}\psi)\sin(\omega_{i}\psi)}{\sin^2(\psi)}  d \psi =\frac{2}{\pi} \int_{0}^{\pi} \sin(\psi)\sin(\omega_{i}\psi) d \psi =  \delta_{i}^{0}	
\end{align*}
we conclude that the sum equals to zero since $i \geq 1$.
\begin{flushleft}
	\underline{Case 4: $i,j,k =0$}
\end{flushleft}
This is the first non-trivial case. We have $C_{000}^{(0)} = 1$ and we will compute
\begin{align*}
	 -  \sum_{\substack{\mu,\nu,\rho=1 \\ \mu + \nu + \rho= \lambda  }}^{\lambda-2} \int_{0}^{\tau} \sin(\tau-s)f_{\mu}^{(0)}(s)f_{\nu}^{(0)}(s)f_{\rho}^{(0)}(s) ds. 
\end{align*}
As before, we define
\begin{align*}
	d_{\mu}^{\alpha}:=
	\begin{cases}
		\alpha=0:& f_{\mu}^{(0)}(0)+ c_{\mu}^{0} \\
		\alpha \geq 1:& c_{\mu}^{\alpha}
	\end{cases}
\end{align*}
and $f_{\mu}^{(0)}(\tau)$ can be written in a more compact form as
\begin{align*}
		f_{\mu}^{(0)}(\tau)=
            f_{\mu}^{(0)}(0) \cos(\tau) + \sum_{\alpha=0}^{\left[\frac{\mu-1}{2} \right]} c_{\mu}^{\alpha} \cos((1+2\alpha)\tau) = \sum_{\alpha=0}^{\left[\frac{\mu-1}{2} \right]}  d_{\mu}^{\alpha} \cos((1+2\alpha)\tau).
\end{align*}
Similarly for $f_{\nu}^{(0)}(\tau)$ and $f_{\rho}^{(0)}(\tau)$. Then,
\begin{align*}
	& -  \sum_{\substack{\mu,\nu,\rho=1 \\ \mu + \nu + \rho= \lambda  }}^{\lambda-2} \int_{0}^{\tau} \sin(\tau-s)f_{\mu}^{(0)}(s)f_{\nu}^{(0)}(s)f_{\rho}^{(0)}(s) ds = \\
	& -  \sum_{\substack{\mu,\nu,\rho=1 \\ \mu + \nu + \rho= \lambda  }}^{\lambda-2} 
	\sum_{\substack{ 0 \leq \alpha_{1} \leq \left[\frac{\mu-1}{2} \right] \\ 
	0 \leq \alpha_{2} \leq \left[\frac{\nu-1}{2} \right] \\
	0 \leq \alpha_{3} \leq \left[\frac{\rho-1}{2} \right]
	} }
	d_{\mu}^{\alpha_{1}}
	d_{\nu}^{\alpha_{2}} d_{\rho}^{\alpha_{3}}
	\int_{0}^{\tau} \sin(\tau-s) \cos((1+2\alpha_{1})s)\cos((1+2\alpha_{2})s)\cos((1+2\alpha_{3})s) ds.
\end{align*}
For any triple $(\alpha_{1},\alpha_{2},\alpha_{3})$ satisfying the conditions above and for any $s \in [0,\tau]$, we split the product into a sum of sinuses as follows
\begin{align*}
	&  8\sin(\tau-s) \cos((1+2\alpha_{1})s)\cos((1+2\alpha_{2})s)\cos((1+2\alpha_{3})s) = \\
	& \sin(\tau+2s(1+\alpha_{1}+\alpha_{2}+\alpha_{3}))+
	\sin(\tau+2s(-\alpha_{1}+\alpha_{2}+\alpha_{3}))+
	\sin(\tau+2s(\alpha_{1}-\alpha_{2}+\alpha_{3}))+\\
	& \sin(\tau+2s(-1-\alpha_{1}-\alpha_{2}+\alpha_{3}))+
	\sin(\tau+2s(\alpha_{1}+\alpha_{2}-\alpha_{3}))+
	\sin(\tau+2s(-1-\alpha_{1}+\alpha_{2}-\alpha_{3}))+\\
	&\sin(\tau+2s(-1+\alpha_{1}-\alpha_{2}-\alpha_{3}))+
	\sin(\tau+2s(-2-\alpha_{1}-\alpha_{2}-\alpha_{3}))
\end{align*}
and observe that all the terms are of the form $\sin(\tau+2s A )$ for various values of $A$ and specifically for $A \in \mathcal{A}(\alpha)$ where
\begin{align*}
    & \alpha=(\alpha_{1},\alpha_{2},\alpha_{3}) \in \mathcal{Q}_{\mu \nu \rho}, \\
	& \mathcal{A}(\alpha):= \big\{ 1+\alpha_{1}+\alpha_{2}+\alpha_{3},  -\alpha_{1}+\alpha_{2}+\alpha_{3},
	\alpha_{1}-\alpha_{2}+\alpha_{3},-1-\alpha_{1}-\alpha_{2}+\alpha_{3}, \\
	&\quad\quad\quad \alpha_{1}+\alpha_{2}-\alpha_{3},-1-\alpha_{1}+\alpha_{2}-\alpha_{3},-1+\alpha_{1}-\alpha_{2}-\alpha_{3},
-2-\alpha_{1}-\alpha_{2}-\alpha_{3}	 \big\}, \\
	& \mathcal{Q}_{\mu \nu \rho}:=\left\{0,1,2,\dots,\left[\frac{\mu-1}{2} \right]  \right\} \times
	\left\{0,1,2,\dots,\left[\frac{\nu-1}{2} \right]  \right\} \times
	\left\{0,1,2,\dots,\left[\frac{\rho-1}{2} \right]  \right\}.
\end{align*}
Using the fact that
\begin{align*}
	\int_{0}^{\tau} \sin(\tau+2s A ) ds =
	\begin{cases}
		A=0:& \tau \sin(\tau),\\
		A \neq 0:& \frac{1}{2A} \left(\cos(\tau)- \cos((1+2A)\tau) \right)
	\end{cases}
\end{align*}
we get
\begin{align*}
	&- \sum_{\substack{\mu,\nu,\rho=1 \\ \mu + \nu + \rho= \lambda  }}^{\lambda-2} 
	\sum_{\alpha \in \mathcal{Q}_{\mu \nu \rho} }
	d_{\mu}^{\alpha_{1}}
	d_{\nu}^{\alpha_{2}} d_{\rho}^{\alpha_{3}}
	\int_{0}^{\tau} \sin(\tau-s) \cos((1+2\alpha_{1})s)\cos((1+2\alpha_{2})s)\cos((1+2\alpha_{3})s) ds = \\
	&- \frac{1}{8} \sum_{\substack{\mu,\nu,\rho=1 \\ \mu + \nu + \rho= \lambda  }}^{\lambda-2} 
	\sum_{\alpha \in \mathcal{Q}_{\mu \nu \rho} }
	d_{\mu}^{\alpha_{1}}
	d_{\nu}^{\alpha_{2}} d_{\rho}^{\alpha_{3}}
	\int_{0}^{\tau} \sum_{A \in \mathcal{A}(\alpha)} \sin(\tau+2s A) ds = \\
	& - \frac{1}{8} \sum_{\substack{\mu,\nu,\rho=1 \\ \mu + \nu + \rho= \lambda  }}^{\lambda-2} 
	\sum_{\alpha \in \mathcal{Q}_{\mu \nu \rho} }
	d_{\mu}^{\alpha_{1}}
	d_{\nu}^{\alpha_{2}} d_{\rho}^{\alpha_{3}}
	\Bigg[ \sum_{ \substack{ A \in \mathcal{A}(\alpha) \\ A \neq 0 }} \frac{1}{2A} \left(\cos(\tau)- \cos((1+2A)\tau) \right) +\sum_{ \substack{ A \in \mathcal{A}(\alpha) \\ A = 0 }}  \tau \sin(\tau) \Bigg],
\end{align*}
which is written in the form
\begin{align*}
	\Theta_{\lambda} \tau \sin(\tau) + \mathcal{R}_{\lambda} \cos(\tau) + \sum_{\beta=0}^{\left[\frac{\lambda-1}{2} \right]} \mathcal{K}_{\lambda}^{\beta} \cos{((1+2\beta)\tau})
\end{align*}
for 
\begin{align}\label{thetadef}
	& \Theta_{\lambda}:= 
	- \frac{1}{8} \sum_{\substack{\mu,\nu,\rho=1 \\ \mu + \nu + \rho= \lambda  }}^{\lambda-2} 
	\sum_{\alpha \in \mathcal{Q}_{\mu \nu \rho} }
	d_{\mu}^{\alpha_{1}}
	d_{\nu}^{\alpha_{2}} 
	d_{\rho}^{\alpha_{3}}
	\sum_{ \substack{ A \in \mathcal{A}(\alpha) \\ A = 0 }} 1,
	\\
	& \mathcal{R}_{\lambda}:= 
	 - \frac{1}{16} \sum_{\substack{\mu,\nu,\rho=1 \\ \mu + \nu + \rho= \lambda  }}^{\lambda-2} 
	\sum_{\alpha \in \mathcal{Q}_{\mu \nu \rho} }
	d_{\mu}^{\alpha_{1}}
	d_{\nu}^{\alpha_{2}} 
	d_{\rho}^{\alpha_{3}}
	\sum_{ \substack{ A \in \mathcal{A}(\alpha) \\ A \neq 0 }} \frac{1}{A} \nonumber
\end{align}
and for some constant $\mathcal{K}_{\lambda}^{\beta}$ provided that, for all $(\mu,\nu\,\rho)$ with $\mu,\nu,\rho=1,2,\dots,\lambda-2$ and $\mu+\nu+\rho = \lambda$ as well as for all $\alpha \in \mathcal{Q}_{\mu \nu \rho}$ and $A \in \mathcal{A}(\alpha)$, the range of the absolute value of $1+2A$ is a subset of the range of $1+2\beta$ for $\beta=0,1,2,\dots,\left[\frac{\lambda-1}{2}\right]$. Since there are only eight elements in $\mathcal{A}(\alpha)$ we prove this statement for each one of them. However, we illustrate the proof only for one of the worst case scenarios, namely for $A=-2-\alpha_{1}-\alpha_{2}-\alpha_{3}$. In the case where $\mu,\nu$ are even and $\rho$ is odd, we get $\lambda$ odd since $\mu+\nu+\rho = \lambda$. Hence,
\begin{align*}
	\left[\frac{\mu-1}{2} \right] = \frac{\mu}{2}-1,\quad 
	\left[\frac{\nu-1}{2} \right] = \frac{\nu}{2}-1,\quad
	\left[\frac{\rho-1}{2} \right] = \frac{\rho-1}{2}, \quad
	\left[\frac{\lambda-1}{2} \right] = \frac{\lambda-1}{2}.
\end{align*}
Then,
\begin{align*}
	|1+2A| &= |1-4-2\alpha_{1}-2\alpha_{2}-2\alpha_{3}| \\ 
	& \leq 3+2\left( \alpha_{1}+\alpha_{2}+\alpha_{3} \right) \\
	& \leq 3 + 2 \left( 
	\left[\frac{\mu-1}{2} \right]+
	\left[\frac{\nu-1}{2} \right]+
	\left[\frac{\rho-1}{2} \right]
	\right) \\
	& \leq 3 + 2 \left( 
	\frac{\mu}{2}-1+
	\frac{\nu}{2}-1+
	\frac{\lambda-1}{2}
	\right) \\
	& =3+  \mu+\nu+\rho - 5 
	 = \lambda-2 \\
	& = 1+ 2 \frac{\lambda-3}{2} 
	 \leq  1+ 2 \frac{\lambda-1}{2} 
	 =  1+ 2 \left[\frac{\lambda-1}{2} \right].
\end{align*}
Similarly, one can prove the validity of the statement above for all the other cases.
\begin{flushleft}
	\underline{\text{Summary of Step 3A.}}
\end{flushleft}
In conlcusion, we get
	\begin{align*}
	f_{\lambda}^{(0)}(\tau)&= f_{\lambda}^{(0)}(0) \cos(\tau) + \sum_{\beta=0}^{\left[\frac{\lambda-1}{2} \right]} \mathcal{A}_{\lambda}^{\beta} \cos{((1+2\beta)\tau}) +\Bigg[\Theta_{\lambda} +  \sum_{\substack{\mu,\nu=0 \\ \mu + \nu= \lambda \\ (\mu,\nu) \neq (\lambda,0) }}^{\lambda} \frac{\theta_{\nu} }{2}
	 \left(f_{\mu}^{(0)}(0)+c_{\mu}^{0} \right) \Bigg] \tau \sin(\tau) 
\end{align*}
for some constant $\mathcal{A}_{\lambda}^{\gamma}$ depending only on the $c_{\mu}^{\alpha}$ where $\Theta_{\lambda}$ is defined in \eqref{thetadef}. In this expression, all the $\theta_{0},\theta_{1},\dots,\theta_{\lambda-2}$ are already fixed due to the induction hypothesis whereas the $\theta_{\lambda-1}$ is still a free variable to be chosen. To this end, we use the fact that $f_{1}^{(0)}(0)=1$ and $c_{1}^{0}=0$ from Step 1 to simply write
\begin{align*}
	\Theta_{\lambda} + \sum_{\substack{\mu,\nu=0 \\ \mu + \nu= \lambda \\ (\mu,\nu) \neq (\lambda,0) }}^{\lambda} \frac{\theta_{\nu} }{2}
	 \left(f_{\mu}^{(0)}(0)+c_{\mu}^{0} \right) 
\end{align*}
as
\begin{align*}	 
\Theta_{\lambda} +	 \frac{\theta_{\lambda-1} }{2}
	 + \sum_{\substack{\mu,\nu=0 \\ \mu + \nu= \lambda \\ (\mu,\nu) \neq (\lambda,0) \\ (\mu,\nu) \neq  (1,\lambda-1)  }}^{\lambda} \frac{\theta_{\nu} }{2}
	 \left(f_{\mu}^{(0)}(0)+c_{\mu}^{0} \right)   
\end{align*}
and hence we choose 
\begin{align*}
	\frac{\theta_{\lambda-1}}{2} = 
	- \Bigg[ \Theta_{\lambda} + \sum_{\substack{\mu,\nu=0 \\ \mu + \nu= \lambda \\ (\mu,\nu) \neq (\lambda,0) \\ (\mu,\nu)  \neq (1,\lambda-1)  }}^{\lambda} \frac{\theta_{\nu} }{2}
	 \left(f_{\mu}^{(0)}(0)+c_{\mu}^{0} \right)\Bigg]
\end{align*}
so that
\begin{align*}
	f_{\lambda}^{(0)}(\tau) &=f_{\lambda}^{(0)}(0) \cos(\tau) + \sum_{\beta=0}^{\left[\frac{\lambda-1}{2} \right]} \mathcal{A}_{\lambda}^{\beta} \cos{((1+2\beta)\tau}).
\end{align*}
\begin{flushleft}
	\underline{\text{Step 3B}: Proof for $\xi=\lambda$ and $m \geq 1$.}
\end{flushleft}
From the recurrence relation \eqref{telikoreccurence1}-\eqref{telikoreccurence2}, each $f_{\lambda}^{(m)}(\tau)$ with $m \geq 1$ is given in terms of the $f_{\xi}^{(i)}(\tau)$ with $\xi=0,1,2,\dots,\lambda-1$ and $i=0,1,2,\dots$ by
\begin{align*}
	f_{\lambda}^{(m)}(\tau)  &= 
	f_{\lambda}^{(m)}(0)\cos(\tau) + \sum_{\substack{\mu,\nu=0 \\ \mu + \nu= \lambda \\ (\mu,\nu) \neq (\lambda,0)}}^{\lambda} \theta_{\nu} 
	\left(
	f_{\mu}^{(m)}(0)\cos(\tau) - f_{\mu}^{(m)}(\tau)
	\right)
	+ \int_{0}^{\tau} \sin(\omega_{m}(\tau-s)) N_{\lambda}^{(m)}(s) ds, \\ 
N_{\lambda}^{(m)}(s)&:=
\omega_{m}\sum_{\substack{\mu,\nu=0 \\ \mu + \nu= \lambda \\ (\mu,\nu) \neq (\lambda,0) }}^{\lambda} \theta_{\nu} f_{\mu}^{(m)}(s)
-\frac{1}{\omega_{m}}\sum_{i,j,k,=0}^{\infty} C_{ijk}^{(m)} \sum_{\substack{\mu,\nu,\rho=1 \\ \mu + \nu + \rho= \lambda  }}^{\lambda-2} f_{\mu}^{(i)}(s)f_{\nu}^{(j)}(s)f_{\rho}^{(k)}(s). 
\end{align*}
First, we observe that
\begin{align*}
	\sum_{\substack{\mu,\nu=0 \\ \mu + \nu= \lambda \\ (\mu,\nu) \neq (\lambda,0)}}^{\lambda} \theta_{\nu} 
	\left(
	f_{\mu}^{(m)}(0)\cos(\tau) - f_{\mu}^{(m)}(\tau) 
	\right)=0
\end{align*}
since now $m \geq 1$ and so $f_{\mu}^{(m)}(\tau)$ is simply given by $f_{\mu}^{(m)}(\tau) = f_{\mu}^{(m)}(0)\cos(\tau)$ for all such $\mu$. Second, we compute the first part of of the integral of $\sin(\omega_{m}(\tau-s))N_{\lambda}^{(m)}(s)$ namely
\begin{align*}
	\omega_{m}\sum_{\substack{\mu,\nu=0 \\ \mu + \nu= \lambda \\ (\mu,\nu) \neq (\lambda,0) }}^{\lambda} \theta_{\nu} \int_{0}^{\tau} \sin(\omega_{m}(\tau-s)) f_{\mu}^{(m)}(s) ds.
\end{align*}
As before, notice that this quantity depends only on $f_{\xi}^{(m)}(\tau)$, $\xi =0,1,2,\dots, \lambda -1$ and so
\begin{align*}
	&\omega_{m} \sum_{\substack{\mu,\nu=0 \\ \mu + \nu= \lambda \\ (\mu,\nu) \neq (\lambda,0) }}^{\lambda} \theta_{\nu} \int_{0}^{\tau} \sin(\omega_{m}(\tau-s))f_{\mu}^{(m)}(s) ds =\\
	&\omega_{m} \sum_{\substack{\mu,\nu=0 \\ \mu + \nu= \lambda \\ (\mu,\nu) \neq (\lambda,0) }}^{\lambda} \theta_{\nu} f_{\mu}^{(m)}(0) \int_{0}^{\tau} \sin(\omega_{m}(\tau-s))\cos(\omega_{m}s) ds =\\
	&\omega_{m}\Bigg[ \sum_{\substack{\mu,\nu=0 \\ \mu + \nu= \lambda \\ (\mu,\nu) \neq (\lambda,0) }}^{\lambda} \theta_{\nu} f_{\mu}^{(m)}(0) \Bigg] \frac{1}{2} \tau \sin(\omega_{m}\tau).
\end{align*}
However, we assumed that $f_{\xi-2}^{(m)}(0) = 0$ for all $m =1,2,3,\dots$ and $\xi=2,3,\dots,\lambda-1$ which is equivalent to $f_{\xi}^{(m)}(0) = 0$ for all $m =1,2,3,\dots$ and $\xi=0,1,\dots,\lambda-3$. So, the only case where the sum is not zero due to the choice of the initial data is when $(\mu,\nu)\in \{(\lambda-2,2),(\lambda-1,1) \}$ and the choice $(\mu,\nu) = (\lambda-1,1)$ does not contribute to the sum because $\theta_{1}=0$. Hence, the sum above boils down to only one secular term
\begin{align*}
\omega_{m}   \frac{\theta_{2}}{2} f_{\lambda-2}^{(m)}(0)  \tau \sin(\omega_{m}\tau).
\end{align*} 
Third, we focus on the second part of the integral of $\sin(\omega_{m}(\tau-s))N_{\lambda}^{(m)}(s)$, namely on
\begin{align}\label{symmquant2}
	  -\sum_{i,j,k,=0}^{\infty} C_{ijk}^{(m)} \sum_{\substack{\mu,\nu,\rho=1 \\ \mu + \nu + \rho= \lambda  }}^{\lambda-2} \int_{0}^{\tau} \sin(\omega_{m}(\tau-s))f_{\mu}^{(i)}(s)f_{\nu}^{(j)}(s)f_{\rho}^{(k)}(s) ds.
\end{align}
This quantity depends only on $f_{\xi}^{(l)}(\tau)$, $\xi =1,2,\dots, \lambda -2$ and $l =0,1,2,\dots$. Since $l=0$ is also included, we split the latter into the following cases.
\begin{flushleft}
	\underline{Case 1: $i,j,k \neq 0$}
\end{flushleft}
Then, 
\begin{align*}
	& -\sum_{\substack{  i,j,k,=0 \\ i,j,k \neq 0 }}^{\infty} C_{ijk}^{(m)} \sum_{\substack{\mu,\nu,\rho=1 \\ \mu + \nu + \rho= \lambda  }}^{\lambda-2} \int_{0}^{\tau} \sin(\omega_{m}(\tau-s))f_{\mu}^{(i)}(s)f_{\nu}^{(j)}(s)f_{\rho}^{(k)}(s) ds = \\
	& -\sum_{\substack{  i,j,k=1 }}^{\infty} C_{ijk}^{(m)} \sum_{\substack{\mu,\nu,\rho=1 \\ \mu + \nu + \rho= \lambda  }}^{\lambda-2}f_{\mu}^{(i)}(0)f_{\nu}^{(j)}(0)f_{\rho}^{(k)}(0) \int_{0}^{\tau} \sin(\omega_{m}(\tau-s))\cos(\omega_{i}s)\cos(\omega_{j}s)\cos(\omega_{k}s) ds.
\end{align*}
As before, we assumed that $f_{\xi-2}^{(m)}(0) = 0$ for all $m =1,2,3,\dots$ and $\xi=2,3,\dots,\lambda-1$ which is equivalent to $f_{\xi}^{(m)}(0) = 0$ for all $m =1,2,3,\dots$ and $\xi=0,1,\dots,\lambda-3$. So, the only case where this sum is not zero due to the choice of the initial data is when $\mu = \nu = \rho = \lambda -2$. Then, $\lambda = \mu + \nu + \rho = 3( \lambda -2)$ which holds only when $\lambda=3$ and we initially assumed that $\lambda \geq 4$. In conclusion, this sum equals to zero.
\begin{flushleft}
	\underline{Case 2: ($i=0$ and $j,k \neq 0$) or ($j=0$ and $i,k \neq 0$) or ($k=0$ and $i,j \neq 0$)}
\end{flushleft}
Again, the quantity inside the sum in \eqref{symmquant2} is symmetric with respect to $i,j,k$. 
Due to this symmetry, we only consider one of these cases. In particular, we assume that $i=0$ and $j,k \neq 0$. Then, we have 
\begin{align*}
	& -\sum_{\substack{  j,k,=0 \\ j,k \neq 0 }}^{\infty} C_{0jk}^{(m)} \sum_{\substack{\mu,\nu,\rho=1 \\ \mu + \nu + \rho= \lambda  }}^{\lambda-2} \int_{0}^{\tau} \sin(\omega_{m}(\tau-s))f_{\mu}^{(0)}(s)f_{\nu}^{(j)}(s)f_{\rho}^{(k)}(s) ds = \\
	& -\sum_{\substack{  j,k=1 }}^{\infty} 
	C_{0jk}^{(m)} \sum_{\substack{\mu,\nu,\rho=1 \\ \mu + \nu + \rho= \lambda  }}^{\lambda-2} \int_{0}^{\tau} \sin(\omega_{m}(\tau-s)) \Bigg[
	f_{\mu}^{(0)}(0) \cos(s) + \sum_{\alpha=0}^{\left[\frac{\mu-1}{2} \right]} c_{\mu}^{\alpha} \cos((1+2\alpha)s) \Bigg] \cdot \\
	&\quad \quad\quad\quad\quad\quad\quad\quad\quad\quad\quad\quad \cdot f_{\nu}^{(j)}(0)f_{\rho}^{(k)}(0)\cos(\omega_{j}s) \cos(\omega_{k}s) 
	ds.
\end{align*}
The $f_{\mu}^{(0)}(0)$ are all free variables. As before, we assumed that $f_{\xi-2}^{(m)}(0) = 0$ for all $m =1,2,3,\dots$ and $\xi=2,3,\dots,\lambda-1$ which is equivalent to $f_{\xi}^{(m)}(0) = 0$ for all $m =1,2,3,\dots$ and $\xi=0,1,\dots,\lambda-3$. So, the only case where this sum is not zero due to the choice of the initial data is when $\nu = \rho = \lambda -2$. Then, $\lambda = \mu + \nu + \rho = \mu + 2(\lambda -2)$  which implies $\mu+\lambda=4$ and we initially assumed that $\lambda \geq 4$. Hence, the only possible case is $\mu=0$ only for $\lambda =4$ which contradicts the condition $\mu \geq 1$. In conclusion, the sum equals to zero.
\begin{flushleft}
	\underline{Case 3: ($i \neq 0$ and $j,k =0$) or ($j \neq 0$ and $i,k = 0$) or ($k \neq 0$ and $i,j = 0$)}
\end{flushleft}
Again, due to the symmetry, we only consider one of these cases. In particular, we assume that $i \neq 0$ and $j,k = 0$. Then, we compute
\begin{align*}
C_{i00}^{(m)}=\frac{2}{\pi} \int_{0}^{\pi} \frac{\sin^2(\omega_{0}\psi)\sin(\omega_{i}\psi)\sin(\omega_{m}\psi)}{\sin^2(\psi)}  d \psi =\frac{2}{\pi} \int_{0}^{\pi} \sin(\omega_{i}\psi)\sin(\omega_{m}\psi) d \psi =  \delta_{i}^{m}
\end{align*}
and
\begin{align*}
	& -\sum_{\substack{  i,j,k=0 \\ i\neq 0, j,k=0}}^{\infty} C_{ijk}^{(m)} \sum_{\substack{\mu,\nu,\rho=1 \\ \mu + \nu + \rho= \lambda  }}^{\lambda-2} \int_{0}^{\tau} \sin(\omega_{m}(\tau-s)) f_{\mu}^{(i)}(s)f_{\nu}^{(j)}(s)f_{\rho}^{(k)}(s) ds = \\
	& -\sum_{\substack{  i=1 }}^{\infty} C_{i00}^{(m)} \sum_{\substack{\mu,\nu,\rho=1 \\ \mu + \nu + \rho= \lambda  }}^{\lambda-2} \int_{0}^{\tau}  \sin(\omega_{m}(\tau-s)) f_{\mu}^{(i)}(s)f_{\nu}^{(0)}(s)f_{\rho}^{(0)}(s) ds = \\
	& - \sum_{\substack{\mu,\nu,\rho=1 \\ \mu + \nu + \rho= \lambda  }}^{\lambda-2} \int_{0}^{\tau}  \sin(\omega_{m}(\tau-s)) f_{\mu}^{(m)}(s)f_{\nu}^{(0)}(s)f_{\rho}^{(0)}(s) ds = \\
	& - \sum_{\substack{\mu,\nu,\rho=1 \\ \mu + \nu + \rho= \lambda  }}^{\lambda-2} 
	f_{\mu}^{(m)}(0)
	\int_{0}^{\tau}  \sin(\omega_{m}(\tau-s)) \cos(\omega_{m}s)
	\Bigg[
	f_{\nu}^{(0)}(0)\cos(s) + \sum_{\alpha=0}^{\left[\frac{\nu-1}{2} \right]} c_{\nu}^{\alpha} \cos((1+2\alpha)s)
	\Bigg] \cdot \\
	&\quad \quad\quad\quad\quad\quad\quad\quad\quad\quad\quad\quad \cdot \Bigg[
	f_{\rho}^{(0)}(0)\cos(s) + \sum_{\beta=0}^{\left[\frac{\rho-1}{2} \right]} c_{\rho}^{\beta} \cos((1+2\beta)s)
	\Bigg] ds.
\end{align*}
As before, we assumed that $f_{\xi-2}^{(m)}(0) = 0$ for all $m =1,2,3,\dots$ and $\xi=2,3,\dots,\lambda-1$ which is equivalent to $f_{\xi}^{(m)}(0) = 0$ for all $m =1,2,3,\dots$ and $\xi=0,1,\dots,\lambda-3$. So, the only case where this sum is not zero due to the choice of the initial data is when $\mu=  \lambda -2$. Then, $\lambda = \mu + \nu + \rho = \lambda -2 + \nu + \rho$  which implies $\nu + \rho=2$. However, $\nu,\rho \neq 0$ and hence the only possible choice is $(\nu,\rho)=(1,1)$. Hence, the sum above boils down to
 \begin{align*}
 	 - f_{\lambda-2}^{(m)}(0)
	\int_{0}^{\tau}  \sin(\omega_{m}(\tau-s)) \cos(\omega_{m}s)
	\Bigg[
	f_{1}^{(0)}(0)\cos(s) + c_{1}^{0} \cos(s)
	\Bigg]^2 ds.
 \end{align*} 
Recall from Step 1 that $f_{1}^{(0)}(0)=1$ and $c_{1}^{0}=0$. Therefore, the sum above is simplified even more to
\begin{align*}
	 - f_{\lambda-2}^{(m)}(0)
	\int_{0}^{\tau}  \sin(\omega_{m}(\tau-s)) \cos(\omega_{m}s)
	\cos^2(s)  ds
\end{align*}
from which a secular term occurs since
\begin{align*}
	& \int_{0}^{\tau} \sin(\omega_{m}(\tau-s)) \cos(\omega_{m}s)
	\cos^2(s)  ds = \\ 
	&-\frac{\omega_{m}}{8(\omega_{m}^2-1)} \cos(\omega_{m}\tau)
	+ \frac{\omega_{m}}{16(\omega_{m}-1)} \cos((\omega_{m}-2)\tau)\\
	&-\frac{\omega_{m}}{16(\omega_{m}+1)} \cos((\omega_{m}+2)\tau)
	+\frac{1}{4} \tau \sin(\omega_{m} \tau),
\end{align*}
for all $m \geq 1$.
\begin{flushleft}
	\underline{Case 4: $i,j,k =0$}
\end{flushleft}
As before, we compute $C_{000}^{(m)} = \delta_{0}^{m}=0$ for all $m\geq 1$ and hence the sum
\begin{align*}
	 &  -\sum_{\substack{ i,j,k,=0\\ i,j,k=0  } }^{\infty} C_{ijk}^{(m)} \sum_{\substack{\mu,\nu,\rho=1 \\ \mu + \nu + \rho= \lambda  }}^{\lambda-2} \int_{0}^{\tau} \sin(\omega_{m}(\tau-s))  f_{\mu}^{(i)}(s)f_{\nu}^{(j)}(s)f_{\rho}^{(k)}(s) ds = \\
	 &-C_{000}^{(m)} \sum_{\substack{\mu,\nu,\rho=1 \\ \mu + \nu + \rho= \lambda  }}^{\lambda-2} \int_{0}^{\tau}\sin(\omega_{m}(\tau-s))  f_{\mu}^{(0)}(s)f_{\nu}^{(0)}(s)f_{\rho}^{(0)}(s) ds
\end{align*}
vanishes.
\begin{flushleft}
	\underline{\text{Summary of Step 3B.}}
\end{flushleft}
In conclusion, we get that 
\begin{align*}
	f_{\lambda}^{(m)}(\tau)  &= 
	f_{\lambda}^{(m)}(0)\cos(\tau) 
	+ \omega_{m}   \frac{\theta_{2}}{2} f_{\lambda-2}^{(m)}(0)  \tau \sin(\omega_{m}\tau) \\
	&  - f_{\lambda-2}^{(m)}(0)
	\Bigg[-\frac{\omega_{m}}{8(\omega_{m}^2-1)} \cos(\omega_{m}\tau)
	+ \frac{\omega_{m}}{16(\omega_{m}-1)} \cos((\omega_{m}-2)\tau)\\
	& -\frac{\omega_{m}}{16(\omega_{m}+1)} \cos((\omega_{m}+2)\tau)
	+\frac{1}{4} \tau \sin(\omega_{m} \tau) \Bigg] \\
	& = f_{\lambda}^{(m)}(0)\cos(\tau) + f_{\lambda-2}^{(m)}(0) \Bigg[ \frac{\omega_{m}}{8(\omega_{m}^2-1)} \cos(\omega_{m}\tau)
	- \frac{\omega_{m}}{16(\omega_{m}-1)} \cos((\omega_{m}-2)\tau)\\
	& + \frac{\omega_{m}}{16(\omega_{m}+1)} \cos((\omega_{m}+2)\tau)+
	\left( \omega_{m}  \frac{\theta_{2}}{2} - \frac{1}{4} \right)\tau \sin(\omega_{m} \tau),
	\Bigg]
\end{align*}
for all $m \geq 1$. Recall Step 1 where $\theta_{2}=\frac{3}{4}$. Finally, since
\begin{align*}
	 \omega_{m}  \frac{\theta_{2}}{2} - \frac{1}{4}  = 
	  \frac{3}{8} \omega_{m} - \frac{1}{4} =
	  \frac{1}{8} \left(1+3 m \right) \neq 0,
\end{align*}
we are forced to choose $f_{\lambda-2}^{(m)}(0)=0$ as required to conclude
\begin{align*}
	f_{\lambda}^{(m)}(\tau) = 
	f_{\lambda}^{(m)}(0)\cos(\tau) 
\end{align*}
for all $m \geq 1$.
\end{proof}
Finally, we prove Theorem \ref{maintheorem2}.
\begin{proof}[Proof of Theorem \ref{maintheorem2}]
	Multiplying both sides of \eqref{ODE} with the time derivative of $G$ yields
\begin{align*}
	\frac{d}{d t}
	\left(
	 \left( \frac{d G(t)}{ dt } \right)^2 + G^2(t) + \frac{1}{2}G^4(t) 
	\right) = 0
\end{align*}
which implies the conservation of energy
\begin{align}\label{energyetimate}
	\left( \frac{d G(t)}{ dt } \right)^2 + G^2(t) + \frac{1}{2}G^4(t) = G^2(0) + \frac{1}{2}G^4(0) 
\end{align}
since $G^{\prime}(0)=0$. We set $X=G$ and $Y=G^{\prime}$ we write \eqref{ODE} as a dynamical system
\begin{align}\label{DS}
	\begin{cases}
		X^{\prime}(t)= Y(t), \\
		Y^{\prime}(t)= -X(t) \left( 1+ X^2(t)\right).
	\end{cases}
\end{align}
The solution defines a curve is the phase space $(X,Y)$ parametrized by $(X(t),Y(t))$ with tangent vector $(X^{\prime}(t),Y^{\prime}(t))$. The stationary points are points on the curve such that $(X^{\prime},Y^{\prime})=(0,0)$ and we see that the only (real) stationary point is $(0,0)$. The quantity
\begin{align*}
	\mathcal{E}(X(t),Y(t)):= Y^2(t)  + X^2(t) + \frac{1}{2}X^4(t) 
\end{align*}
is constant on each orbit $(X(t),Y(t))$ and equals to
\begin{align*}
	\mathcal{E}(X(t),Y(t)) = X_{0}^2 + \frac{1}{2}X_{0}^4
\end{align*}
where $X_{0}=X(0)$. We have that, for all $a \neq 0$, the level set
\begin{align*}
    \mathcal{E}_{a}&:=
	\left \{
	(X,Y) \in \mathbb{R}^2: Y^2  + X^2 + \frac{1}{2}X^4 = a^2 
	\right \}
\end{align*}
is connected, compact and $1-$dimensional sub-manifold of $\mathbb{R}^2$. Indeed, for $a \neq 0$, one can verify that the map
\begin{align*}
	\gamma_{a}: \mathbb{R} \longrightarrow \mathcal{E}_{a}:=\gamma_{a} \left(
	\mathbb{R}
	\right) \subset \mathbb{R}^2
\end{align*}
defined by
\begin{align*}
	\gamma_{a}(t) = \left(
	\frac{a \sqrt{2}  \cos(t)}{\sqrt{1+\sqrt{1+2a^2 \cos^2(t)}}},
	a \sin(t)
	\right),\quad t \in \mathbb{R}.
\end{align*}
is a smooth parametrization of the curve. For all $t \in \mathbb{R}$, we compute
\begin{align*}
	\left| \gamma_{a}^{\prime} (t) \right|^2 & =
	a^2 \left(
	\cos^2(t) + 
	\frac{1}{2} 
	\left(
	\frac{1}{1+2a^2 \cos^2(t)}+
	\frac{1}{\sqrt{1+2a^2 \cos^2(t)}}
	\right)
	\sin^2(t)
	\right)
\end{align*}
which cannot be zero so $\gamma_{a}\left(\mathbb{R} \right):=\mathcal{E}_{a}$ is a $1-$dimensional sub-manifold of $\mathbb{R}^2$. Clearly, the curve is closed since
\begin{align*}
	\gamma_{a}(t) = \gamma_{a}(t+2\pi),\quad  \forall t \in \mathbb{R}
\end{align*}
and so
\begin{align*}
	\mathcal{E}_{a}:=\gamma_{a}\left(
	\mathbb{R}
	\right) =
	\gamma_{a}\left(
	[0,2\pi] \cap \mathbb{R}
	\right).
\end{align*}
Since $[0,2\pi]$ is connected subset of $\mathbb{R}$ and $\gamma_{a}$ is continuous we get that its image is also connected. The fact that the image is a compact set also follows from the continuity of $\gamma_{a}$ and the compactness of $[0,2\pi]\cap \mathbb{R}$.  Furthermore, for $a \neq 0$, the unique stationary point $(0,0)$ does not belong in $\mathcal{E}_{a}$. Hence, we obtain that for all $a \neq 0$, $\mathcal{E}_{a}$ is homomorphic to a circle and constitutes a periodic orbit. Since the choice $a=0$ is equivalent to $X_{0}=0$, we conclude that for all $X_{0} \neq 0$ the corresponding orbit is periodic. The following picture illustrates the periodic orbits in the phrase space.
\begin{figure}[h]
    \centering
    \includegraphics[width=0.6\textwidth]{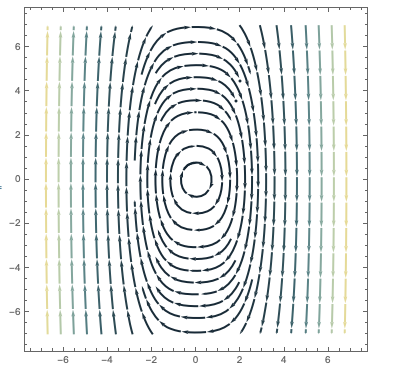}
    \caption{Phase space of the dynamical system \eqref{DS} and time periodic solutions.}
    \label{fig:mesh1}
\end{figure}
\end{proof}
\appendix
\section{Local well-posedness in the energy space}\label{LWP}
In the following lines, we prescribe initial data $(f_{0},f_{1})$ on the $\{t=0\}$ slice and establish the local well-posedness in the energy space for the initial value problem
\begin{align}\label{systemforLWP}
\begin{cases}
	& - \partial_{t}^2 f(t,\psi) +  \partial_{\psi}^2 f(t,\psi) = \frac{f^3(t,\psi)}{\sin^2(\psi)}, \quad (t,\psi) \in I \times (0,\pi), \\
	& f(0,\psi) = f_{0}(\psi),~\partial_{t}f(0,\psi) = f_{1}(\psi),\quad \psi \in (0,\pi) \\
	& f(t,0) = f(t,\pi) = 0,\quad t \in I 
\end{cases}
\end{align}
where $I$ is an interval in $\mathbb{R}$ with $0 \in I$.
The essential ingredients are the embeddings $L^{6}\left(\mathbb{S}^3 \right) \hookrightarrow L^{p}\left(\mathbb{S}^3 \right)$ valid for all $p \in \left[1,6 \right]$ together with the Sobolev inequality on $\mathbb{S}^3$ (Theorem 4.19, page 80, \cite{MR1481970}) 
\begin{align}\label{SobolevInequality}
\|u \|_{L^{p}\left(\mathbb{S}^3 \right)  } \lesssim	\|u \|_{L^{6}\left(\mathbb{S}^3 \right)  } \lesssim \|\nabla u \|_{L^{2}\left(\mathbb{S}^3 \right)}+
\|u \|_{L^{2}\left(\mathbb{S}^3 \right)},\quad  \forall p \in \left[1,6 \right], \quad \forall u \in H^{1}\left(\mathbb{S}^3 \right)
\end{align}
which is written in terms of functions $u=u(\psi,\theta,\phi)$ on the 3-sphere. Once restricted to rotationally symmetric functions $u(\psi,\theta,\phi)=v(\psi)$ and set $f(\psi)=\sin(\psi)v(\psi)$ as above, these estimates yield the following Hardy-Sobolev inequalities
\begin{align}\label{Hardywithp}
\left(	\int_{0}^{\pi} \left|f(\psi) \right|^{p} \sin^{2-p}(\psi)  d \psi  \right)^{\frac{1}{p}}  \lesssim \left(	\int_{0}^{\pi} \left|f^{\prime}(\psi) \right|^{2}  d \psi  \right)^{\frac{1}{2}} + \left(	\int_{0}^{\pi} \left|f(\psi) \right|^{2}  d \psi  \right)^{\frac{1}{2}}, 
\end{align}
valid for all $p \in \left[1,6 \right]$ and $f \in H^{1}[0,\pi]$. Indeed,
\begin{align*}
	\|u \|_{L^{p}\left(\mathbb{S}^3 \right)  }^p 
	& = \int_{0}^{\pi}\int_{0}^{\pi}\int_{0}^{2\pi} \left|u(\psi,\theta,\phi) \right|^{p} \sin^2(\psi) \sin(\theta) d \psi d\theta d\phi \\
	& \simeq \int_{0}^{\pi} \left|v(\psi) \right|^{p} \sin^2(\psi)  d \psi \\
	& \simeq \int_{0}^{\pi} \left|f(\psi) \right|^{p} \sin^{2-p}(\psi)  d \psi, \\
	\| \nabla u \|_{L^{2}\left(\mathbb{S}^3 \right)  }^2 
	& = \int_{0}^{\pi}\int_{0}^{\pi}\int_{0}^{2\pi} \left|\nabla u(\psi,\theta,\phi) \right|^{2} \sin^2(\psi) \sin(\theta) d \psi d\theta d\phi \\
	& \simeq \int_{0}^{\pi} \left| v^{\prime}(\psi) \right|^{2} \sin^2(\psi)  d \psi \\
	& \simeq \int_{0}^{\pi} \left|f^{\prime}(\psi)-\frac{\cos(\psi)}{\sin(\psi)}f(\psi) \right|^{2}  d \psi \\
	& \lesssim \int_{0}^{\pi} \left|f^{\prime}(\psi) \right|^2 d\psi + \int_{0}^{\pi}\left| \frac{f(\psi)}{\sin(\psi)} \right|^{2}  d \psi \\
	& \lesssim \int_{0}^{\pi} \left|f^{\prime}(\psi) \right|^2 d\psi
 \end{align*}
where we used the standard Hardy inequality at the last step together with the boundary conditions in \eqref{systemforLWP}. For any Schwartz function $f \in \mathcal{S}\left( I \times (0,\pi) \right)$, one can multiply the equation with $\partial_{t} f$ and integrate by parts to obtain the conservation of energy
\begin{align*}
	E(f[t]) = E(f[0]), 
\end{align*}
for all times, where the energy reads
\begin{align*}
	E(f[t]):=\frac{1}{2}  \int_{0}^{\pi} \left(
	 \left|\partial_{t} f(t,\psi) \right|^2 +
	  \left|\partial_{\psi} f(t,\psi) \right|^2
	 + \frac{1}{2} \left| \frac{f^2(t,\psi)}{\sin(\psi)}
	 \right|^2 ~\right)d \psi
\end{align*}
and $f[t]:=(f(t,\cdot),\partial_{t}f(t,\cdot))$ denotes the full state of $f(t,\cdot)$, that is position and velocity at time $t$. The conserved energy distinguishes the energy space 
\begin{align*}
	H^{1}[0,\pi] \times L^2[0,\pi]
\end{align*}
that is the function space of initial data $f[0]$ for which the energy is finite. Indeed, the first two terms in 
\begin{align*}
	E(f[0])=\frac{1}{2}  \int_{0}^{\pi} \left(
	 \left|f_{1}(\psi) \right|^2 +
	  \left| f_{0}^{\prime}(\psi) \right|^2
	 + \frac{1}{2} \left| \frac{f_{0}^2(\psi)}{\sin(\psi)}
	\right|^2~ \right) d \psi
\end{align*}
are clearly bounded for all initial data $f[0]=(f_{0},f_{1}) \in H^{1} [0,\pi]\times L^2[0,\pi]$. The third term is also bounded by the $H^{1}-$norm of $f_{0}$ by \eqref{Hardywithp} with $p=4$. 
To establish the local well-posedness for the initial value problem \eqref{systemforLWP} in the energy space, we first define 
\begin{align*}
	f(t,\psi)=\cos(t |\nabla|)f_{0}(\psi) + 
		\frac{\sin(t |\nabla|)}{|\nabla|}f_{1}(\psi)
\end{align*}
the solution to the free wave equation
\begin{align*}
	\begin{cases}
	& - \partial_{t}^2 f(t,\psi) +  \partial_{\psi}^2 f(t,\psi) = 0, \quad (t,\psi) \in I \times (0,\pi), \\
	& f(0,\psi) = f_{0}(\psi),~\partial_{t}f(0,\psi) = f_{1}(\psi),\quad \psi \in (0,\pi) \\
\end{cases}
\end{align*}
where
\begin{align*}
\cos(t |\nabla|)f_{0}(\psi):= \frac{1}{2} \left(
f_{0}(\psi+t) +f_{0}(\psi-t)
\right)  
\end{align*}
and
\begin{align*}
\frac{\sin(t |\nabla|)}{|\nabla|}f_{1}(\psi):=
\frac{1}{2} \int_{\psi-t}^{\psi+t} f_{1}(y) dy 
\end{align*}
are the wave propagators acting on $f_{0}$ and $f_{1}$ respectively.
For all $t$ such that $0\leq \psi \pm t \leq \pi$ for all $\psi \in [0,\pi]$, the following bounds hold
\begin{align} \label{wavepropagator}
	 \left \| \cos(t |\nabla|)f_{0} \right \|_{H^{1}[0,\pi]} \lesssim 
	 \left \| f_{0} \right \|_{H^{1}[0,\pi]}
	,\quad \left \| \frac{\sin(t |\nabla|)}{|\nabla|}f_{1}\right \|_{H^{1}[0,\pi]} \lesssim  \left \| f_{1} \right \|_{L^2[0,\pi]}.
\end{align}
Second, we define the notion of a weak solution to the non-linear problem: a function $f=f(t,\psi)$ is a weak solution to the initial value problem \eqref{systemforLWP} if and only if
	\begin{align*}
		f = \mathcal{K}_{(f_{0},f_{1})}(f)
	\end{align*}
	where
	\begin{align*}
		\mathcal{K}_{(f_{0},f_{1})}(f)(t,\psi)= \
		\cos(t |\nabla|)f_{0}(\psi) + 
		\frac{\sin(t |\nabla|)}{|\nabla|}f_{1}(\psi) - \int_{0}^{t} \frac{\sin((t-s) |\nabla|)}{|\nabla|} \frac{f^3(s,\psi)}{\sin^2(\psi)} ds.
	\end{align*}
Finally, we define the function space $(X_{T},\| \cdot \|_{X_{T}})$ where
	\begin{align*}
		X_{T}:= C \left( 
		[0,T], H^{1}[0,\pi]\right),\quad \| f \|_{X_{T}}:= \sup_{t \in [0,\pi]} \|f(t,\cdot) \|_{H^{1}[0,\pi]}
	\end{align*}
	and
	\begin{align*}
		\overline{\mathcal{B}}_{X_{T}}(0;R):= \{f \in X_{T}: \| f \|_{X_{T}} \leq R \}
	\end{align*}
the closed ball in $X_{T}$ of radius $R$ centered at zero. Now, we prove the following result.
\begin{prop}
	Let $f[0]=(f_{0},f_{1}) \in H^1 [0,\pi] \times L^2 [0,\pi]$, $R:=c \|f[0] \|_{H^1 [0,\pi] \times L^2 [0,\pi]}$ and $T= c^{-1} R^{-2}$ for a sufficiently large positive constant $c$. Then, there exists a unique element $f \in \overline{\mathcal{B}}_{X_{T}}(0;R)$ such that 
	$f = \mathcal{K}_{ f[0] }(f)$. Furthermore, the map $f[0] \longmapsto f$ is Lipschitz continuous.
\end{prop}
\begin{proof}
	Let $f[0]=(f_{0},f_{1}) \in  H^1 [0,\pi] \times L^2 [0,\pi]$, $R:=c \|f[0] \|_{H^1 [0,\pi] \times L^2 [0,\pi]}$ and $T = c^{-1} R^{-2}$. One can easily show that $(X_{T},\| \cdot \|_{X_{T}})$ is a complete metric space and $\overline{\mathcal{B}}_{X_{T}}(0;R)$ is a closed subset. First, we show that $\mathcal{K}_{ f[0] }$ maps the closed ball to its self. To this end, we assume that $f \in \overline{\mathcal{B}}_{X_{T}}(0;R)$ and use the estimates for the wave propagators from \eqref{wavepropagator} together with the Hardy-Sobolev inequality \eqref{Hardywithp} with $p=6$ to estimate
	\begin{align*}
		\left \|
		\mathcal{K}_{ f[0] }(f)(t,\cdot)
		\right \|_{ H^{1} [0,\pi] } 
		&  \lesssim
		\|f[0] \|_{\dot{H}^1 [0,\pi] \times L^2 [0,\pi]}
+ \int_{0}^{T} 
\left \|
\frac{f^3(s,\cdot)}{\sin^2}
\right \|_{L^2[0,\pi]} ds \\
&  \lesssim
		\frac{R}{c}
+ \int_{0}^{T} 
\left \|
f(s,\cdot)
\right \|_{H^1[0,\pi]}^3 ds \\
&  \lesssim
		\frac{R}{c}
+ T
\left \|
f
\right \|_{X_{T}}^3 
\leq \frac{R}{c} + T R^{3}
 =
\frac{2R}{c}
 \leq
R
\end{align*}
by choosing $c$ sufficiently large. 
Next, we establish a local Lipschitz-type estimate
\begin{align}\label{Lipschitz}
	\sup_{s \in [0,T]} \left \|
		\frac{ 
		f^3(s,\cdot)-
		\tilde{f}^3(s,\cdot)
		}{
		\sin^2
		}  
		\right \|_{ L^2[0,\pi] } \lesssim 
		R^2
		 \left \|
		f-
		\tilde{f}  
		\right \|_{ X_{T} }
\end{align}
for all $f, \tilde{f} \in \overline{\mathcal{B}}_{X_{T}}(0;R)$. This estimate follows from Holder's inequality together with the Hardy-Sobolev inequality \eqref{Hardywithp} with $p=6$. Indeed, for any $s\in [0,T]$ and $f, \tilde{f} \in \overline{\mathcal{B}}_{X_{T}}(0;R)$, we have
\begin{align*}
	& \left \|
		\frac{ 
		f^3(s,\cdot)-
		\tilde{f}^3(s,\cdot)
		}{
		\sin^2
		}  
		\right \|_{ L^2[0,\pi] } = \\
	& \left \|
		\frac{ 
		\left(
		f(s,\cdot)-
		\tilde{f}(s,\cdot)
		\right)
		\left(
		f^2(s,\cdot)+
		\tilde{f}(s,\cdot)f(s,\cdot)+
		\tilde{f}^2(s,\cdot)
		\right)
		}{
		\sin^{\frac{2}{3}} \sin^{\frac{4}{3}}
		}  
		\right \|_{ L^2[0,\pi] } \leq 	\\
	& \left \|
		\frac{ 
		f(s,\cdot)-
		\tilde{f}(s,\cdot)
		}{\sin^{\frac{2}{3}} }
		\right \|_{ L^6[0,\pi] }
		\left \|
		\frac{ 
		f^2(s,\cdot)+
		\tilde{f}(s,\cdot)f(s,\cdot)+
		\tilde{f}^2(s,\cdot)
		}{\sin^{\frac{4}{3}} }
		\right \|_{ L^3[0,\pi] } \leq  	\\
	& \left \|
		\frac{ 
		f(s,\cdot)-
		\tilde{f}(s,\cdot)
		}{\sin^{\frac{2}{3}} }
		\right \|_{ L^6[0,\pi] }
		\left(
		\left \|
		\frac{ 
		f^2(s,\cdot)
		}{\sin^{\frac{4}{3}} }
		\right \|_{ L^3[0,\pi] } +
		\left \|
		\frac{ 
		\tilde{f}(s,\cdot)f(s,\cdot)
		}{\sin^{\frac{2}{3}} \sin^{\frac{2}{3}} }
		\right \|_{ L^3[0,\pi] } +
		\left \|
		\frac{ 
		\tilde{f}^2(s,\cdot)
		}{\sin^{\frac{4}{3}} }
		\right \|_{ L^3[0,\pi] }
		\right)
		\leq \\
	& \left \|
		\frac{ 
		f(s,\cdot)-
		\tilde{f}(s,\cdot)
		}{\sin^{\frac{2}{3}} }
		\right \|_{ L^6[0,\pi] }
		\left(
		\left \|
		\frac{ 
		f^2(s,\cdot)
		}{\sin^{\frac{4}{3}} }
		\right \|_{ L^3[0,\pi] } +
		\left \|
		\frac{ 
		\tilde{f}(s,\cdot)
		}{\sin^{\frac{2}{3}} }
		\right \|_{ L^6[0,\pi] }
		\left \|
		\frac{ 
		f(s,\cdot)
		}{\sin^{\frac{2}{3}} }
		\right \|_{ L^6[0,\pi] } +
		\left \|
		\frac{ 
		\tilde{f}^2(s,\cdot)
		}{\sin^{\frac{4}{3}} }
		\right \|_{ L^3[0,\pi] }
		\right)
		\lesssim \\
	& \left \| 
	f(s,\cdot)- \tilde{f}(s,\cdot)
	\right \|_{H^{1}[0,\pi]}	
	\left(
	\left \| 
	f(s,\cdot)
	\right \|_{H^{1}[0,\pi]}	^2 +
	\left \| 
	f(s,\cdot)
	\right \|_{H^{1}[0,\pi]}	
	\left \| 
	\tilde{f}(s,\cdot)
	\right \|_{H^{1}[0,\pi]}	 +
	\left \| 
	\tilde{f}(s,\cdot)
	\right \|_{H^{1}[0,\pi]}	^2
	\right)	\leq \\
	& \left \| 
	f- \tilde{f}
	\right \|_{ X_{T} }	
	\left(
	\left \| 
	f
	\right \|_{ X_{T} }^2 +
	\left \| 
	f
	\right \|_{X_{T} }	
	\left \| 
	\tilde{f}
	\right \|_{X_{T} }	 +
	\left \| 
	\tilde{f}
	\right \|_{X_{T} }	^2
	\right)	\leq 3 R^2 \left \| 
	f- \tilde{f}
	\right \|_{ X_{T} }.	
\end{align*}
Next, we show that $\mathcal{K}_{f[0]}$ is a contraction map. We pick again two elements $f, \tilde{f} \in \overline{\mathcal{B}}_{X_{T}}(0;R)$ and use the Lipschitz-type bound \eqref{Lipschitz} to estimate
\begin{align*}
	\left \|
		\mathcal{K}_{ f[0] }(f)(t,\cdot)-
		\mathcal{K}_{ f[0] }(\tilde{f} )(t,\cdot)
		\right \|_{ H^{1} [0,\pi] } 
		& \lesssim 
		 \int_{0}^{T} \left \|
		\frac{f^3(s,\cdot)-
		\tilde{f}^3(s,\cdot)}{\sin^2}   
		\right \|_{ L^{2} [0,\pi] } ds \\
	& \lesssim  T R^2  
		 \left \|
		f-
		\tilde{f}  
		\right \|_{ X_{T} }	\leq  \frac{1}{c}  
		 \left \|
		f-
		\tilde{f}  
		\right \|_{ X_{T} }.
\end{align*} 
By choosing $c$ sufficiently large, we get
\begin{align*}
	\left \|
		\mathcal{K}_{ f[0] }
		\right \|_{ \text{Lip} \left( \overline{\mathcal{B}}_{X_{T}}(0;R) \right) }:=
		\sup_{\substack{ f,\tilde{f} \in \overline{\mathcal{B}}_{X_{T}}(0;R) \\ f \neq \tilde{f} }}
		\frac{
		\left \|
		\mathcal{K}_{ f[0] }(f)-
		\mathcal{K}_{ f[0] }(\tilde{f} )
		\right \|_{X_{T} }
		}{
		 \left \|
		f-
		\tilde{f}  
		\right \|_{ X_{T} }
		}
		 \leq \frac{1}{2}.
\end{align*}
Now, by Banach's fixed point theorem, there exists a unique element $f \in \overline{\mathcal{B}}_{X_{T}}(0;R)$ such that 
	$f = \mathcal{K}_{ f[0] }(f)$. Finally, we prove that the map $f[0] \longmapsto f$ is Lipschitz continuous. To this end, we choose two initial data $f[0]=(f_{0},f_{1})$ and $\tilde{f}[0]=(\tilde{f}_{0},\tilde{f}_{1})$ in the energy space which produce two unique solutions existing within the time intervals $[0,T]$ and $[0,\tilde{T}]$ respectively and
	\begin{align*}
		 f = \mathcal{K}_{ f[0] }(f), \quad f  \in \overline{\mathcal{B}}_{X_{T}}(0;R), \quad 
		\tilde{f} = \mathcal{K}_{ \tilde{f}[0] }(\tilde{f}), \quad \tilde{f}  \in \overline{\mathcal{B}}_{X_{T}}(0;\tilde{R})
	\end{align*}
	 for some positive $R$ and $\tilde{R}$. Now, these two solutions coincide in the closed ball $\overline{\mathcal{B}}_{X_{T}}(0;r)$, for $r:= \min \{R,\tilde{R} \}$, due the uniqueness established above. Then, for all $t \in [0,\min \{T,\tilde{T} \}]$, we use \eqref{wavepropagator} to conclude
	 \begin{align*}
	 	\left \|
	 	f(t,\cdot) - \tilde{f}(t,\cdot)
	 	\right \|_{H^{1}[0,\pi]} & =
	 	\left \|
	 	\mathcal{K}_{ f[0] }(f)(t,\cdot) -
	 	\mathcal{K}_{ \tilde{f}[0] }(\tilde{f})(t,\cdot)
	 	\right \|_{H^{1}[0,\pi]} \\
	 	& \leq 
	 	\left \|
	 	\cos(t|\nabla|)(f_{0}-\tilde{f}_{0})
	 	\right \|_{H^{1}[0,\pi]} +
	 	\left \|
	 	\frac{\sin(t|\nabla|)}{|\nabla|} (f_{1}-\tilde{f}_{1})
	 	\right \|_{H^{1}[0,\pi]} \\
	 	& \lesssim
	 	\left \|
	 	f_{0}-\tilde{f}_{0}
	 	\right \|_{H^{1}[0,\pi]} +
	 	\left \|
	 	f_{1}-\tilde{f}_{1}
	 	\right \|_{L^{2}[0,\pi]} \\
	 	& =
	 	\left \|
	 	f[0]-\tilde{f}[0]
	 	\right \|_{H^{1}[0,\pi] \times L^{2}[0,\pi]}.
	 \end{align*} 
\end{proof}
\section{Closed formula for the interaction coefficients}\label{formulaC}
In this section, we establish a necessary closed formula for the interaction coefficients defined in \eqref{FourierConstants} which was used to produce the iterations in section \ref{iterations}. We use the notation
\begin{align*}
\mathbbm{1} \left( \text{condition} \right) = 
	\begin{cases}
	1, \text{ if the condition is satisfied} \\
	0, \text{ otherwise} \\
	\end{cases}
\end{align*}
\begin{lemma}\label{ClosedFormulaC}
For all $i,j,k,m=0,1,2,\dots$ such that $i \leq j$ and $k \leq m$ we have
\begin{align*}
	C_{ijk}^{(m)}= 
	\sum_{\alpha=0}^{i}\sum_{\beta=0}^{k} 
		  \mathbbm{1} \left( 2(\alpha-\beta) = (m-k)-(j-i) \right).
\end{align*}
\begin{proof}
	Pick any $i,j,k,m=0,1,2,\dots$ such that $i \leq j$ and $k \leq m$. Then, we use the definition of the Chebyshev polynomial of second kind,
	\begin{align*}
		 \sin(\omega_{n}\psi) = \sin(\psi) U_{n}(\cos(\psi)),\quad n =0,1,2,\dots
	\end{align*}
together with the analog to the addition theorem
\begin{align*}
U_{p}(y)U_{q}(y)
= \sum_{\substack{  r= q-p \\ \text{step 2} } }^{p+q} U_{r}(y)
=\sum_{s=0}^{p} U_{q-p+2s}(y),\quad p,q =0,1,2,\dots,\quad p \leq q
\end{align*}	
to write
	\begin{align*}
		C_{ijk}^{(m)} & = \frac{2}{\pi} \int_{0}^{\pi} \frac{\sin(\omega_{i}\psi)\sin(\omega_{j}\psi)\sin(\omega_{k}\psi)\sin(\omega_{m}\psi)}{\sin^2(\psi)} d \psi \\
		&= \frac{2}{\pi} \int_{0}^{\pi} \sin^2(\psi)U_{i}(\cos(\psi))U_{j}(\cos(\psi))U_{k}(\cos(\psi))U_{m}(\cos(\psi))  d \psi \\
		&= \frac{2}{\pi} \int_{-1}^{1} U_{i}(y)U_{j}(y)U_{k}(y)U_{m}(y) \sqrt{1-y^2}  d y  \\
		&= \frac{2}{\pi} \int_{-1}^{1} \left( 
		\sum_{\alpha=0}^{i} U_{j-i+2\alpha}(y)
		\right) 
		\left( 
		\sum_{\beta=0}^{k} U_{m-k+2\beta}(y)
		\right)
		 \sqrt{1-y^2}  d y \\
		 &= \frac{2}{\pi}\sum_{\alpha=0}^{i}\sum_{\beta=0}^{k} 
		  \int_{-1}^{1} 
		 U_{j-i+2\alpha}(y)
		U_{m-k+2\beta}(y)
		 \sqrt{1-y^2}  d y \\
		 &= \frac{2}{\pi}\sum_{\alpha=0}^{i}\sum_{\beta=0}^{k} \frac{\pi}{2}
		  \mathbbm{1} \left( j-i+2\alpha = m-k+2\beta \right) \\
		  &= \sum_{\alpha=0}^{i}\sum_{\beta=0}^{k} 
		  \mathbbm{1} \left( 2(\alpha-\beta) = m-k-j+i \right).
	\end{align*}
\end{proof}
\end{lemma}
\begin{remark}
	Note that the formula above and also be written as
	\begin{align*}
		C_{ijk}^{(m)}=
	\sum_{\substack{ p=j-i \\ \text{step }2} }^{j+i}
	\sum_{\substack{ q=m-k \\ \text{step }2} }^{m+k}
       \mathbbm{1} \left( p = q \right),
	\end{align*}
	for all $i,j,k,m=0,1,2,\dots$ such that $i \leq j$ and $k \leq m$.
\end{remark}

\nocite{*}

\begin{thebibliography}{10}

\bibitem{SecondPublication}
Athanasios Chatzikaleas.
\newblock On the Fourier analysis of the Einstein-Klein-Gordon system: Growth and Decay of the Fourier constants.
\newblock {arXiv:2004.11049}.


\bibitem{MR3356988}
Thomas Alazard and Pietro Baldi.
\newblock Gravity capillary standing water waves.
\newblock {\em Arch. Ration. Mech. Anal.}, 217(3):741--830, 2015.

\bibitem{13044166}
Luis~Lehner Alex~Buchel, Steven L.~Liebling.
\newblock Boson stars in ads.

\bibitem{12100890}
Steven L.~Liebling Alex~Buchel, Luis~Lehner.
\newblock Scalar collapse in ads.

\bibitem{MR2639896}
David~M. Ambrose and Jon Wilkening.
\newblock Computation of time-periodic solutions of the {B}enjamin-{O}no
  equation.
\newblock {\em J. Nonlinear Sci.}, 20(3):277--308, 2010.

\bibitem{MR2430631}
Alain Bachelot.
\newblock The {D}irac system on the anti-de {S}itter universe.
\newblock {\em Comm. Math. Phys.}, 283(1):127--167, 2008.

\bibitem{MR3867631}
Pietro Baldi, Massimiliano Berti, Emanuele Haus, and Riccardo Montalto.
\newblock Time quasi-periodic gravity water waves in finite depth.
\newblock {\em Invent. Math.}, 214(2):739--911, 2018.

\bibitem{MR3097022}
Pietro Baldi, Massimiliano Berti, and Riccardo Montalto.
\newblock A note on {KAM} theory for quasi-linear and fully nonlinear forced
  {K}d{V}.
\newblock {\em Atti Accad. Naz. Lincei Rend. Lincei Mat. Appl.},
  24(3):437--450, 2013.

\bibitem{MR3237812}
Pietro Baldi, Massimiliano Berti, and Riccardo Montalto.
\newblock K{AM} for quasi-linear {K}d{V}.
\newblock {\em C. R. Math. Acad. Sci. Paris}, 352(7-8):603--607, 2014.

\bibitem{MR3569244}
Pietro Baldi, Massimiliano Berti, and Riccardo Montalto.
\newblock K{AM} for autonomous quasi-linear perturbations of {K}d{V}.
\newblock {\em Ann. Inst. H. Poincar\'{e} Anal. Non Lin\'{e}aire},
  33(6):1589--1638, 2016.

\bibitem{MR3502158}
Pietro Baldi, Massimiliano Berti, and Riccardo Montalto.
\newblock K{AM} for autonomous quasi-linear perturbations of m{K}d{V}.
\newblock {\em Boll. Unione Mat. Ital.}, 9(2):143--188, 2016.

\bibitem{MR1819863}
D.~Bambusi and S.~Paleari.
\newblock Families of periodic solutions of resonant {PDE}s.
\newblock {\em J. Nonlinear Sci.}, 11(1):69--87, 2001.

\bibitem{MR2345400}
Massimiliano Berti.
\newblock {\em Nonlinear oscillations of {H}amiltonian {PDE}s}, volume~74 of
  {\em Progress in Nonlinear Differential Equations and their Applications}.
\newblock Birkh\"{a}user Boston, Inc., Boston, MA, 2007.

\bibitem{MR3502157}
Massimiliano Berti.
\newblock K{AM} for {PDE}s.
\newblock {\em Boll. Unione Mat. Ital.}, 9(2):115--142, 2016.

\bibitem{MR3073240}
Massimiliano Berti, Luca Biasco, and Michela Procesi.
\newblock Existence and stability of quasi-periodic solutions for derivative
  wave equations.
\newblock {\em Atti Accad. Naz. Lincei Rend. Lincei Mat. Appl.},
  24(2):199--214, 2013.

\bibitem{MR2038735}
Massimiliano Berti and Philippe Bolle.
\newblock Multiplicity of periodic solutions of nonlinear wave equations.
\newblock {\em Nonlinear Anal.}, 56(7):1011--1046, 2004.

\bibitem{MR2248834}
Massimiliano Berti and Philippe Bolle.
\newblock Cantor families of periodic solutions for completely resonant
  nonlinear wave equations.
\newblock {\em Duke Math. J.}, 134(2):359--419, 2006.

\bibitem{MR2592290}
Massimiliano Berti and Philippe Bolle.
\newblock Sobolev periodic solutions of nonlinear wave equations in higher
  spatial dimensions.
\newblock {\em Arch. Ration. Mech. Anal.}, 195(2):609--642, 2010.

\bibitem{MR2813578}
Massimiliano Berti and Philippe Bolle.
\newblock Quasi-periodic solutions of nonlinear {S}chr\"{o}dinger equations on
  {$\Bbb T^d$}.
\newblock {\em Atti Accad. Naz. Lincei Rend. Lincei Mat. Appl.},
  22(2):223--236, 2011.

\bibitem{MR2967117}
Massimiliano Berti and Philippe Bolle.
\newblock Sobolev quasi-periodic solutions of multidimensional wave equations
  with a multiplicative potential.
\newblock {\em Nonlinearity}, 25(9):2579--2613, 2012.

\bibitem{MR2998835}
Massimiliano Berti and Philippe Bolle.
\newblock Quasi-periodic solutions with {S}obolev regularity of {NLS} on {$\Bbb
  T^d$} with a multiplicative potential.
\newblock {\em J. Eur. Math. Soc. (JEMS)}, 15(1):229--286, 2013.

\bibitem{MR3312439}
Massimiliano Berti, Livia Corsi, and Michela Procesi.
\newblock An abstract {N}ash-{M}oser theorem and quasi-periodic solutions for
  {NLW} and {NLS} on compact {L}ie groups and homogeneous manifolds.
\newblock {\em Comm. Math. Phys.}, 334(3):1413--1454, 2015.

\bibitem{MR3625065}
Massimiliano Berti and Riccardo Montalto.
\newblock Quasi-periodic water waves.
\newblock {\em J. Fixed Point Theory Appl.}, 19(1):129--156, 2017.

\bibitem{MR4062430}
Massimiliano Berti and Riccardo Montalto.
\newblock Quasi-{P}eriodic {S}tanding {W}ave {S}olutions of
  {G}ravity-{C}apillary {W}ater {W}aves.
\newblock {\em Mem. Amer. Math. Soc.}, 263(1273):0, 2020.

\bibitem{MR3205859}
Piotr Bizo\'{n}.
\newblock Is {A}d{S} stable?
\newblock {\em Gen. Relativity Gravitation}, 46(5):Art. 1724, 14, 2014.

\bibitem{MR586417}
Ha\"{\i}m Br\'{e}zis, Jean-Michel Coron, and Louis Nirenberg.
\newblock Free vibrations for a nonlinear wave equation and a theorem of {P}.
  {R}abinowitz.
\newblock {\em Comm. Pure Appl. Math.}, 33(5):667--684, 1980.

\bibitem{Choptuik}
M.W. Choptuik.
\newblock {\em Phys. Rev. Lett.}, 70, 1993.

\bibitem{MR1316662}
Demetrios Christodoulou and Sergiu Klainerman.
\newblock {\em The global nonlinear stability of the {M}inkowski space},
  volume~41 of {\em Princeton Mathematical Series}.
\newblock Princeton University Press, Princeton, NJ, 1993.

\bibitem{MR1239318}
Walter Craig and C.~Eugene Wayne.
\newblock Newton's method and periodic solutions of nonlinear wave equations.
\newblock {\em Comm. Pure Appl. Math.}, 46(11):1409--1498, 1993.

\bibitem{DafermosTalk}
M.~Dafermos.
\newblock The black hole stability problem.
\newblock {\em Talk given at the Newton Institute, Cambridge}, 2006.

\bibitem{MR2901562}
Jean-Marc Delort.
\newblock Periodic solutions of nonlinear {S}chr\"{o}dinger equations: a
  paradifferential approach.
\newblock {\em Anal. PDE}, 4(5):639--676, 2011.

\bibitem{MR3002881}
\'{O}scar J.~C. Dias, Gary~T. Horowitz, Don Marolf, and Jorge~E. Santos.
\newblock On the nonlinear stability of asymptotically anti-de {S}itter
  solutions.
\newblock {\em Classical Quantum Gravity}, 29(23):235019, 24, 2012.

\bibitem{MR2978943}
\'{O}scar J.~C. Dias, Gary~T. Horowitz, and Jorge~E. Santos.
\newblock Gravitational turbulent instability of anti-de {S}itter space.
\newblock {\em Classical Quantum Gravity}, 29(19):194002, 7, 2012.

\bibitem{200208393}
Andrzej~Rostworowski Dominika Hunik-Kostyra.
\newblock Ads instability: resonant system for gravitational perturbations of
  ads5 in the cohomogeneity-two biaxial bianchi ix ansatz.

\bibitem{MR1986317}
Lawrence~C. Evans.
\newblock Some new {PDE} methods for weak {KAM} theory.
\newblock {\em Calc. Var. Partial Differential Equations}, 17(2):159--177,
  2003.

\bibitem{MR2496651}
Lawrence~C. Evans.
\newblock Further {PDE} methods for weak {KAM} theory.
\newblock {\em Calc. Var. Partial Differential Equations}, 35(4):435--462,
  2009.

\bibitem{MR868737}
Helmut Friedrich.
\newblock On the existence of {$n$}-geodesically complete or future complete
  solutions of {E}instein's field equations with smooth asymptotic structure.
\newblock {\em Comm. Math. Phys.}, 107(4):587--609, 1986.

\bibitem{MR701918}
G.~W. Gibbons, S.~W. Hawking, Gary~T. Horowitz, and Malcolm~J. Perry.
\newblock Positive mass theorems for black holes.
\newblock {\em Comm. Math. Phys.}, 88(3):295--308, 1983.

\bibitem{MR2551709}
Sean~A. Hartnoll.
\newblock Lectures on holographic methods for condensed matter physics.
\newblock {\em Classical Quantum Gravity}, 26(22):224002, 61, 2009.

\bibitem{MR264959}
S.~W. Hawking and R.~Penrose.
\newblock The singularities of gravitational collapse and cosmology.
\newblock {\em Proc. Roy. Soc. London Ser. A}, 314:529--548, 1970.

\bibitem{MR1481970}
Emmanuel Hebey.
\newblock {\em Sobolev spaces on {R}iemannian manifolds}, volume 1635 of {\em
  Lecture Notes in Mathematics}.
\newblock Springer-Verlag, Berlin, 1996.

\bibitem{MR2913628}
Gustav Holzegel and Jacques Smulevici.
\newblock Self-gravitating {K}lein-{G}ordon fields in asymptotically
  anti-de-{S}itter spacetimes.
\newblock {\em Ann. Henri Poincar\'{e}}, 13(4):991--1038, 2012.

\bibitem{MR3369103}
Gustav Holzegel and Claude~M. Warnick.
\newblock The {E}instein-{K}lein-{G}ordon-{A}d{S} system for general boundary
  conditions.
\newblock {\em J. Hyperbolic Differ. Equ.}, 12(2):293--342, 2015.

\bibitem{11084539}
Piotr~Bizon Joanna~Jalmuzna, Andrzej~Rostworowski.
\newblock A comment on ads collapse of a scalar field in higher dimensions.

\bibitem{MR911772}
S.~B. Kuksin.
\newblock Hamiltonian perturbations of infinite-dimensional linear systems with
  imaginary spectrum.
\newblock {\em Funktsional. Anal. i Prilozhen.}, 21(3):22--37, 95, 1987.

\bibitem{MR1754991}
Sergei~B. Kuksin.
\newblock A {KAM}-theorem for equations of the {K}orteweg-de {V}ries type.
\newblock {\em Rev. Math. Math. Phys.}, 10(3):ii+64, 1998.

\bibitem{MR1857574}
Sergei~B. Kuksin.
\newblock {\em Analysis of {H}amiltonian {PDE}s}, volume~19 of {\em Oxford
  Lecture Series in Mathematics and its Applications}.
\newblock Oxford University Press, Oxford, 2000.

\bibitem{MR2070057}
Sergei~B. Kuksin.
\newblock Fifteen years of {KAM} for {PDE}.
\newblock In {\em Geometry, topology, and mathematical physics}, volume 212 of
  {\em Amer. Math. Soc. Transl. Ser. 2}, pages 237--258. Amer. Math. Soc.,
  Providence, RI, 2004.

\bibitem{MR0344645}
Peter~D. Lax.
\newblock Periodic solutions of the {K}d{V} equations.
\newblock In {\em Nonlinear wave motion ({P}roc. {AMS}-{SIAM} {S}ummer {S}em.,
  {C}larkson {C}oll. {T}ech., {P}otsdam, {N}.{Y}., 1972)}, pages 85--96.
  Lectures in Appl. Math., Vol. 15, 1974.

\bibitem{MR369963}
Peter~D. Lax.
\newblock Periodic solutions of the {K}d{V} equation.
\newblock {\em Comm. Pure Appl. Math.}, 28:141--188, 1975.

\bibitem{MR404889}
Peter~D. Lax.
\newblock Almost periodic solutions of the {K}d{V} equation.
\newblock {\em SIAM Rev.}, 18(3):351--375, 1976.

\bibitem{MR2842962}
Jianjun Liu and Xiaoping Yuan.
\newblock A {KAM} theorem for {H}amiltonian partial differential equations with
  unbounded perturbations.
\newblock {\em Comm. Math. Phys.}, 307(3):629--673, 2011.

\bibitem{DafermosHolzegel}
G.~Holzegel M.~Dafermos.
\newblock Dynamic instability of solitons in $4+1$ dimesnional gravity with
  negative cosmological constant.
\newblock 2006.

\bibitem{13033186}
Andrzej~Rostworowski Maciej~Maliborski.
\newblock Time-periodic solutions in einstein ads - massless scalar field
  system,
\newblock {arXiv:1303.3186}.  

\bibitem{MR1633016}
Juan Maldacena.
\newblock The large {$N$} limit of superconformal field theories and
  supergravity.
\newblock {\em Adv. Theor. Math. Phys.}, 2(2):231--252, 1998.

\bibitem{MR1705508}
Juan Maldacena.
\newblock The large-{$N$} limit of superconformal field theories and
  supergravity.
\newblock volume~38, pages 1113--1133. 1999.
\newblock Quantum gravity in the southern cone (Bariloche, 1998).

\bibitem{MR3109357}
Maciej Maliborski and Andrzej Rostworowski.
\newblock Turbulent instability of anti-de {S}itter space-time.
\newblock {\em Internat. J. Modern Phys. A}, 28(22-23):1340020, 12, 2013.

\bibitem{09090518}
John McGreevy.
\newblock Holographic duality with a view toward many-body physics.

\bibitem{MR3603787}
Riccardo Montalto.
\newblock Quasi-periodic solutions of forced {K}irchhoff equation.
\newblock {\em NoDEA Nonlinear Differential Equations Appl.}, 24(1):Art. 9, 71,
  2017.

\bibitem{181204268}
Georgios Moschidis.
\newblock A proof of the instability of ads for the einstein--massless vlasov
  system.

\bibitem{170408681}
Georgios Moschidis.
\newblock A proof of the instability of ads for the einstein--null dust system
  with an inner mirror.

\bibitem{11043702}
Andrzej~Rostworowski Piotr~Bizon.
\newblock On weakly turbulent instability of anti-de sitter space.

\bibitem{Bizon1}
Oleg Evnin Dominika Hunik Vincent Luyten Maciej~Maliborski Piotr~Bizon,
  Ben~Craps.
\newblock Conformal flow on s3 and weak field integrability in ads4.
\newblock {\em Comm. Math. Phys.}, 353:1179--1199, 2017.

\bibitem{MR470378}
Paul~H. Rabinowitz.
\newblock Free vibrations for a semilinear wave equation.
\newblock {\em Comm. Pure Appl. Math.}, 31(1):31--68, 1978.

\bibitem{0501128}
Alfonso~V. Ramallo.
\newblock Introduction to the ads/cft correspondence.

\bibitem{MR612249}
Richard Schoen and Shing~Tung Yau.
\newblock Proof of the positive mass theorem. {II}.
\newblock {\em Comm. Math. Phys.}, 79(2):231--260, 1981.

\bibitem{MR0106295}
Gabor Szeg\"{o}.
\newblock {\em Orthogonal polynomials}.
\newblock American Mathematical Society Colloquium Publications, Vol. 23.
  Revised ed. American Mathematical Society, Providence, R.I., 1959.

\bibitem{MR1040892}
C.~Eugene Wayne.
\newblock Periodic and quasi-periodic solutions of nonlinear wave equations via
  {KAM} theory.
\newblock {\em Comm. Math. Phys.}, 127(3):479--528, 1990.

\bibitem{MR626707}
Edward Witten.
\newblock A new proof of the positive energy theorem.
\newblock {\em Comm. Math. Phys.}, 80(3):381--402, 1981.


\bibitem{MR4026950}
Piotr Bizo\'{n}, Dominika Hunik-Kostyra, and Dmitry Pelinovsky.
\newblock Stationary states of the cubic conformal flow on {$\Bbb S^3$}.
\newblock {\em Discrete Contin. Dyn. Syst.}, 40(1):1--32, 2020.

\end{thebibliography}

\end{document}